\documentclass[paper=a4, fontsize=11pt]{scrartcl} 

\makeatletter 
\def\blfootnote{\gdef\@thefnmark{}\@footnotetext}
\makeatother

\usepackage[T1]{fontenc} 
\usepackage[english]{babel} 
\usepackage{amsmath,amsfonts,amsthm, verbatim, textcomp, amssymb,mathtools, empheq, color, multicol,units, xfrac} 
\usepackage{enumitem}
\usepackage{tikz-cd}
\usepackage{sectsty} 

\sectionfont{\centering \normalfont\scshape\large} 
\subsectionfont{\normalfont\bfseries}
\subsubsectionfont{\normalfont\bfseries}

\makeatother

\usepackage[margin=3.1cm,footskip=25pt,headheight=20pt]{geometry}
\usepackage{fancyhdr}

\usepackage{hyperref}

\numberwithin{equation}{section} 
\numberwithin{figure}{section} 
\numberwithin{table}{section} 

\newcommand\isoto{\xrightarrow{\,\smash{\raisebox{-0.65ex}{\ensuremath{\scriptstyle\sim}}}\,}}
\newcommand\longisoto{\xrightarrow{\;\,\smash{\raisebox{-0.65ex}{\ensuremath{\scriptstyle\sim}}}\;\,}}

\newcommand{\Hom}{\text{Hom}}
\newcommand{\End}{\text{End}}
\newcommand{\Ext}{\text{Ext}}
\newcommand{\Ker}{\text{Ker}}

\newcommand{\into}{\hookrightarrow}

\newcommand{\ur}{\underline{r}}
\newcommand{\ud}{\underline{\epsilon}}
\newcommand{\us}{\underline{s}}
\newcommand{\uinfty}{\underline{\infty}}

\newcommand{\Ind}{\text{Ind}}
\newcommand{\Res}{\text{Res}}
\newcommand{\tr}{\text{tr}}
\newcommand{\cato}{\mathcal{O}}
\newcommand{\catO}{\mathcal{O}}
\newcommand{\mf}[1]{\mathfrak{#1}}
\newcommand{\mc}[1]{\mathcal{#1}}
\newcommand{\onto}{\twoheadrightarrow}

\newcommand{\ZZ}{\mathbb{Z}}

\newcommand{\NN}{\mathbb{N}}

\newcommand{\QQ}{\mathbb{Q}}

\newcommand{\parity}{\epsilon}

\newtheoremstyle{all}
  {11pt}
  {11pt}
  {\slshape}
  {}
  {\bfseries}
  {}
  {.5em}
  {}

\theoremstyle{all}
\newtheorem{stable}{Definition}[section]
\newtheorem{hinfty}[stable]{Definition}
\newtheorem{pinftydef}[stable]{Definition}
\newtheorem*{conv}{Conventions.}

\newtheorem{trunc}[stable]{Proposition}
\newtheorem{prinj}[stable]{Lemma}

\newtheorem{hecketensor}[stable]{Proposition}
\newtheorem{Uprop}[stable]{Theorem}

\newtheorem{Lem4.2}[stable]{Lemma}
\newtheorem{++hwc}[stable]{Proposition}
\newtheorem{tilt}[stable]{Proposition}
\newtheorem{tradj}[stable]{Lemma}

\newtheorem{trproj}[stable]{Proposition}

\newtheorem{pinftythm}[stable]{Theorem}
\newtheorem{O++schurcat}[stable]{Corollary}
\newtheorem{stablemult}[stable]{Proposition}
\newtheorem{finranktpc}[stable]{Theorem}
\newtheorem{Ejleft}[stable]{Definition}
\newtheorem{phi11}[stable]{Lemma}
\newtheorem{Ejright}[stable]{Definition}
\newtheorem{Fjadj}[stable]{Proposition}
\newtheorem{infranktpc}[stable]{Theorem}

\newtheorem{HeckeDef}[stable]{Definition}
\newtheorem{CategoricalAction}[stable]{Definition}
\newtheorem{TPCdef}[stable]{Definition}
\newtheorem{combinatorialsuperduality}[stable]{Proposition}
\newtheorem{TPCsuperduality}[stable]{Proposition}
\newtheorem{uniquenessofTPCs}[stable]{Theorem}
\newtheorem{QHA}[stable]{Remark}
\newtheorem{Lem4.2rmk}[stable]{Remark}
\newtheorem{gradedlifts}[stable]{Theorem}
\newtheorem{gradedsuperduality}[stable]{Corollary}
\newtheorem{superduality}[stable]{Corollary}
\newtheorem{catOfinitewedges}[stable]{Remark}
\newtheorem{projcover}[stable]{Remark}
\newtheorem{def:trunc}[stable]{Definition}
\newtheorem{Rmk:Chenproj}[stable]{Remark}

\newtheorem{EjLinO}[stable]{Lemma}
\newtheorem{limitinO}[stable]{Lemma}
\newtheorem{rmk:deltamultstable}[stable]{Remark}
\newtheorem{rmk:rMassump}[stable]{Remark}
\newtheorem{equivcatOs}[stable]{Remark}
\newtheorem{not:Mr}[stable]{Notation}
\newtheorem{rmk:prinjectives}[stable]{Remark}
\newtheorem{rmk:wedgenotationcomp}[stable]{Remark}
\newtheorem{wts}[stable]{Remark}
\newtheorem*{acknowledgements}{Acknowledgements.}

\begin{document}

\blfootnote{2010 \textit{Mathematics Subject Classification}: 17B10.

\noindent \textit{Key words and phrases}: Lie superalgebras, categorification}

\begin{center}
	\LARGE Graded super duality for general linear Lie superalgebras \\
	\vspace{13pt}
	\Large Christopher Leonard
	\vspace{13pt}
	\end{center}

\begin{abstract}
	\noindent \textsc{Abstract}: We provide a new proof of the super duality equivalence between infinite-rank parabolic BGG categories of general linear Lie (super) algebras conjectured by Cheng and Wang and first proved by Cheng and Lam. We do this by establishing a new uniqueness theorem for tensor product categorifications motivated by work of Brundan, Losev, and Webster. Moreover we show that these BGG categories have Koszul graded lifts and super duality can be lifted to a graded equivalence.
	\end{abstract}


\section{Introduction}\label{sec:intro}

In \cite[Conjecture~4.32]{Bru02a}, Brundan formulated a Kazhdan-Lusztig style (BKL) conjecture on the characters of irreducible modules for the complex general linear Lie superalgebra $\mf{gl}_{m|n}$. The idea was that characters in the integral category $\catO$ for $\mf{gl}_{m|n}$ are controlled by certain canonical bases in the $\mf{sl}_{\ZZ}$-module $V^{\otimes m}\otimes W^{\otimes n}$, where $V=\mathbb{C}^{\ZZ}$ is the natural $\mf{sl}_{\ZZ}$-module and $W=V^*$ is its restricted dual. This conjecture admits natural parabolic variants (see e.g. \cite[Conjecture 3.10]{CW07}). The corresponding $\mf{sl}_{\ZZ}$-modules are of the form
\begin{equation}\label{eq:wedgemodule}
	\bigwedge\nolimits^{\underline{m}}V\otimes\bigwedge\nolimits^{\underline{n}}W:=\bigwedge\nolimits^{m_1} V\otimes\cdots\otimes \bigwedge\nolimits^{m_s} V\otimes\bigwedge\nolimits^{n_1} W\otimes\cdots\otimes \bigwedge\nolimits^{n_t} W
	\end{equation}
	
\noindent where $\underline{m}=(m_1,\ldots, m_s)$, $\underline{n}=(n_1,\ldots, n_t)$, $m=\sum m_i$, and $n=\sum n_j$ (in fact the wedges of $V$s and $W$s can be in any order, but we restrict our attention to this case for the introduction).

There is a natural $\mf{sl}_{\ZZ}$-module isomorphism $\bigwedge^{\infty}V\cong\bigwedge^{\infty}W$ between semi-infinite wedge spaces. In \cite{CW07} the authors considered the induced isomorphism
\begin{equation}\label{eq:combSD}
	\bigwedge\nolimits^{\underline{m}}V\otimes\bigwedge\nolimits^{\underline{n}}W\otimes\bigwedge\nolimits^{\infty}V\cong \bigwedge\nolimits^{\underline{m}}V\otimes\bigwedge\nolimits^{\underline{n}}W\otimes\bigwedge\nolimits^{\infty}W.
	\end{equation}
	
\noindent Motivated by this, they defined infinite-rank parabolic categories $\catO$, denoted $\catO^{++}_0$ and $\catO^{++}_1$ respectively, for the corresponding infinite-rank general linear Lie superalgebras and conjectured an equivalence $\catO^{++}_0\cong\catO^{++}_1$. This is super duality (the special case $s=1$, $t=0$ was first conjectured in \cite{CWZ06}).

Super duality was proved in \cite[Theorem~5.1]{CL09}. It can be regarded as a bridge between categories $\catO$ for general linear Lie algebras and superalgebras. Its utility stems from `truncation functors' (see Definition \ref{def:trunc} or \cite{CW07}) from the categories $\catO^{++}_{\parity}$ (${\parity\in\{0,1\}}$) to finite-rank categories $\catO$. These allow us to read off information about the finite-rank categories from their infinite-rank counterparts. For example, the BKL conjecture for the special case $t=1$ is an immediate corollary of super duality, as observed in \cite[Remark~4.19]{CW07}. Super duality was also a key component in the first general proof of the BKL conjecture in \cite[Theorem~8.11]{CLW12}.

Starting with an arbitrary Kac-Moody algebra $\mf{g}$ and a tensor product $M$ of integrable highest weight modules for $\mf{g}$, Losev and Webster \cite{LW13} defined what it means for a categorical action of $\mf{g}$ on a category $\mc{C}$ (in the sense of Chuang and Rouquier \cite{CR04}, \cite{Rou08}) to be a \emph{tensor product categorification} (TPC) of $M$. Building on uniqueness for categorifications of integrable highest weight modules in \cite{Rou12} they showed that such TPCs are unique up to equivalence.

For $r\in \NN$, let $I_r=\{i\in\ZZ\mid 1-r\leq i\leq r-1\}$ and let $\mf{sl}_{I_r}$ be the special linear Lie algebra of traceless complex matrices with rows and columns indexed by $I_r^+:=I_r\cup(I_r+1)$. Let $V_r$ be the natural $\mf{sl}_{I_r}$-module and $W_r=V_r^*$ its dual. 

In a highly non-trivial extension of the uniqueness result of \cite{LW13}, Brundan, Losev, and Webster proved in \cite[Theorem~2.12]{BLW13} proved that TPCs of $\mf{sl}_{\ZZ}$-modules of the form \eqref{eq:wedgemodule} are unique up to equivalence. Roughly, they did this by regarding the $\mf{sl}_{\ZZ}$-module as a union of $\mf{sl}_{I_r}$-modules:
\begin{equation}\label{eq:moduleunion}
	\bigwedge\nolimits^{\underline{m}}V\otimes\bigwedge\nolimits^{\underline{n}}W=\bigcup_{r=1}^{\infty}\bigwedge\nolimits^{\underline{m}}V_r\otimes\bigwedge\nolimits^{\underline{n}}W_r.
	\end{equation}

\noindent Moreover, they showed in \cite[Theorem~3.10]{BLW13} that translation functors on a finite-rank integral parabolic category $\catO$ for $\mf{gl}_{m|n}$ make the category a TPC of the corresponding $\mf{sl}_{\ZZ}$-module (\ref{eq:wedgemodule}) and that any such TPC possesses a unique Koszul graded lift compatible with its categorification structures. 

This paper examines TPCs of modules of the form \eqref{eq:combSD} and applications to the infinite-rank categories $\catO^{++}_{\parity}$. In Theorem~\ref{thm:uniquenessofTPCs} we show these TPCs are unique up to equivalence. We deal with semi-infinite wedges by regarding them as a direct limit of finite wedges. So the decomposition \eqref{eq:moduleunion} is replaced by the direct limit
\begin{equation}
	\bigwedge\nolimits^{\underline{m}}V\otimes\bigwedge\nolimits^{\underline{n}}W\otimes\bigwedge\nolimits^{\infty}V=\lim_{\longrightarrow}~ \bigwedge\nolimits^{\underline{m}}V_r\otimes\bigwedge\nolimits^{\underline{n}}W_r\otimes\bigwedge\nolimits^rV_r,
	\end{equation}

\noindent or the analogous limit for $\bigwedge^{\infty} W$.

Take $\parity\in\{0,1\}$. In Theorem~\ref{infranktpc} we prove that the infinite-rank category $\catO^{++}_{\parity}$ is a TPC of the corresponding $\mf{sl}_{\ZZ}$-module. Super duality follows immediately. Our main tools are the truncation functors already mentioned. These allow us to consider a module in $\catO^{++}_{\parity}$ as a direct limit of modules for finite-rank general linear Lie superalgebras as the rank goes to infinity. First we need to establish that $\catO^{++}_{\parity}$ has enough projectives. This is a new result of this paper. We construct projective covers in $\catO^{++}_{\parity}$ as the direct limit of projective covers in certain subcategories of finite-rank categories $\catO$. The key property is that the Verma flags of these projective covers stabilize when the rank is sufficiently large.

An unexpected difficulty occurs in defining the functors $E_j$ on $\catO^{++}_{\parity}$ lifting the action of the Chevalley generators $e_j$.  The corresponding translation functors on finite-rank categories $\catO$ are not well behaved with respect to truncation, so we can't define connecting maps for the direct limit in the obvious way. We actually define two sets of connecting maps, leading to isomorphic functors $E_j^{\texttt{L}}$ and $E_j^{\texttt{R}}$ which are naturally left and right adjoint to $F_j$ respectively. Again, these only give well-defined functors on $\catO^{++}_{\parity}$ because the composition multiplicities eventually stabilize.

Finally we replace $\mf{sl}_{\ZZ}$-modules with modules over the quantum group $\mc{U}_q\mf{sl}_{\ZZ}$ and categories with graded categories. Following \cite[Section~5]{BLW13} we show that our $\mf{sl}_{\ZZ}$-TPCs have unique graded lifts that are $\mc{U}_q\mf{sl}_{\ZZ}$-TPCs, and these graded categories are Koszul. In particular, the categories $\catO^{++}_{\parity}$ have Koszul graded lifts and super duality lifts to a graded equivalence between them. This is a new result of this paper.

It is our hope that this work will be useful in the study of the representation theory of quantum supergroups at roots of unity.

This paper is organized as follows. In Section \ref{section:combinatorics} we describe the combinatorics of the wedge spaces. In Section \ref{section:uniquenessofTPCs} we define our TPCs and show that they are unique up to equivalence. In Section \ref{section:catO} we define the infinite-rank categories $\catO^{++}_{\parity}$, show they have enough projectives, and define the categorical actions on them. We prove they are TPCs and deduce the super duality equivalence. Finally in Section \ref{section:gradings} we establish the existence of graded lifts.
	
\begin{conv}
	By a \emph{Schurian category} we mean an abelian category $\mc{C}$ such that all objects are of finite length, there are enough projectives and injectives, and the endomorphism algebras of the irreducible objects are one dimensional. Let $K_0\left(\mathcal{C}\right)$ denote the split Grothendieck group of the full subcategory of projective objects in $\mc{C}$. Set $[ \mathcal{C}] := \mathbb{C}\otimes_{\mathbb{Z}} K_0(\mathcal{C})$.
		
	If $A$ is an associative algebra we write mod-$A$ for the category of finite-dimensional right $A$-modules. If we have a collection of objects $X_r$ indexed by $\mathbb{N}\cup \{\infty\}$ we will drop the subscript in the $r=\infty$ case and write $X=X_\infty$. If $\mf{a}$ is a Lie superalgebra we will say `$\mf{a}$-module' to mean $\mf{a}$-supermodule.
	\end{conv}
	
\begin{acknowledgements}
	The author would like to thank his advisor Weiqiang Wang for suggesting this project and for his generosity with his time and advice.
	\end{acknowledgements}

\section{Combinatorics}\label{section:combinatorics}

Take  $r\in \mathbb{N}\cup \{\infty\}$. Let 
\begin{align}
	I_r:=\{ \, i\in\ZZ \, \mid \, 1-r\leq i\leq r-1 \, \}
	\end{align}
	
\noindent and $I_r^+:=I_r\cup(I_r+1)$. Let $\mathfrak{sl}_{I_r}$ denote the Lie algebra of complex traceless matrices with rows and columns indexed by $I_r^+$ and with only finitely many non-zero entries in each row and column. Let $f_i:=e_{i,i+1}$ (resp. $e_i:=e_{i+1,i}$) denote the matrix unit with 1 in position $(i,\, i+1)$ (resp. $(i+1,\, i)$) and 0s elsewhere. Define the weight and root lattices of $\mf{sl}_{I_r}$ by
\begin{equation}
	P_r:=\bigoplus_{i\in I_r} \mathbb{Z} \omega_i \qquad \text{and} \qquad Q_r:=\bigoplus_{i\in I_r} \mathbb{Z} \alpha_i
	\end{equation}
	
\noindent respectively, where $\alpha_i:=2\omega_i-\omega_{i-1}-\omega_{i+1}$ and we interpret $\omega_i$ as 0 if $i\notin I_r$. We will denote the non-degenerate pairing $P_r \times Q_r \to \mathbb{Z}$ by $\left(\omega,\alpha\right)\mapsto \omega\cdot\alpha$.

Let  $P_r^+\subseteq P_r$ (resp. $Q_r^+\subseteq Q_r$) denote the $\mathbb{Z}_{\geq 0}$-span of the $\omega_i$ (resp. $\alpha_i$). Define the dominance ordering on $P_r$ by declaring that $\beta\geq \gamma$ if $\beta-\gamma\in Q_r^+$. For $i\in I_r^+$ let $\varepsilon_i:=\omega_i-\omega_{i-1}$ where again we interpret $\omega_i$ as 0 if $i\notin I_r$. 

Let $V_r$ be the natural $\mathfrak{sl}_{I_r}$-module of column vectors with standard basis $\{v_i\mid i\in I_r^+\}$. Let $W_r$ be the dual of $V_r$ with dual basis $\{w_i\mid i\in I_r^+\}$. We will denote the exterior powers $\bigwedge^n V_r$ and $\bigwedge^n W_r$ by $\bigwedge^{n,0} V_r$ and $\bigwedge^{n,1} V_r$, respectively. For $n\in \mathbb{N}$ and $\parity\in\{0,1\}$ let $\Xi_{r;n,\parity}$ be the set of 01-tuples $\lambda=(\lambda_i)_{i\in I_r^+}$ such that
\begin{align}
	\left| \left\{ \, i \in I_r^+ \, \middle| \,  \lambda_i\neq \parity \, \right\} \right| = n. 
	\end{align}

\noindent This set indexes the monomial basis of $\bigwedge^{n,\parity} V_r$: if $\lambda\in \Xi_{r;n,\parity}$ and $i_1>\cdots >i_n$ with  ${\lambda_{i_1},\ldots,\lambda_{i_n}\neq \parity}$ then set
\begin{align} \label{eq:vlam}
	v_\lambda:=\begin{cases}
				v_{i_1}\wedge \cdots \wedge v_{i_n}		& \text{if } \parity=0 \\
				w_{i_n}\wedge \cdots \wedge w_{i_1}	& \text{if } \parity=1.
				\end{cases}
	\end{align}

\noindent The weight of $v_\lambda$ is
\begin{align}
	\left| \lambda\right| := \sum_{\substack{i\in I_r^+ \\ \lambda_i\neq \parity}} (-1)^{\parity} \varepsilon_i\in P_r.
	\end{align}
	
\noindent We wish to consider spaces of semi-infinite wedges. We will construct these as direct limits of the modules above. For $r<s<\infty$ there is an $\mathfrak{sl}_{I_r}$-module embedding ${\bigwedge\nolimits^{r,\parity} V_r\into \bigwedge\nolimits^{s,\parity} V_s}$ given by linearly extending the assignment
\begin{align} \label{wedgemap}
	v_\lambda\longmapsto\begin{cases}
							v_\lambda\wedge v_{-r}\wedge v_{-r-1}\wedge\cdots \wedge v_{1-s} 		& \text{ if } \parity=0 \\
							v_\lambda\wedge w_{r+1}\wedge w_{r+2}\wedge \cdots \wedge w_s		& \text{ if } \parity=1.
							\end{cases}
	\end{align}

\noindent The corresponding embedding $\Xi_{r;r,\parity}\into \Xi_{s;s,\parity}$ is given by extending a 01-tuple ${\lambda=\left(\lambda_j\right)_{j\in I_r^+}}$ by setting
\begin{align}\label{indexsetwedgemap}
	\lambda_j = 
	\begin{cases}
		0	& \text{ if } r<j\leq s \\
		1	& \text{ if } 1-s\leq j<1-r.
		\end{cases}
	\end{align}
	
\noindent Let $\Xi_{\infty,\parity}:=\displaystyle{\lim_{\longrightarrow}} \, \Xi_{r;r,\parity}$ and let $\bigwedge\nolimits^{\infty,\parity} V:= \displaystyle{\longrightarrow}\textstyle \, \bigwedge\nolimits^{r,\parity} V_r$, an $\mathfrak{sl}_\mathbb{Z}$-module. Then $\Xi_{\infty,\parity}$ parameterizes the natural monomial basis of $\bigwedge\nolimits^{\infty,\parity} V$, the union of the monomial bases of the $\bigwedge\nolimits^{r,\parity} V_r$. It is natural to think of an element $\lambda\in\Xi_{\infty,\parity}$ as a 01-tuple $\lambda=\left(\lambda_i\right)_{i\in \mathbb{Z}}$ and of the corresponding element $v_\lambda\in\bigwedge\nolimits^{\infty,\parity} V$ as a semi-infinite wedge.

\vspace{10pt}
\begin{center}
	\framebox(350,40){\emph{Once and for all fix $l\in\mathbb{N}$, $n_1,\ldots, n_l\in \mathbb{N}$, and $c_1,\ldots, c_l\in \left\{ \, 0,1\right\}$.}}
	\end{center}
\vspace{10pt}

We will generally suppress dependence on this data in our notation.  For $r\in\mathbb{N}\cup\{\infty\}$ and $\parity\in\left\{0,1\right\}$ write $\ur=\left(n_1,\ldots, n_l, r\right)$ and $\ud=\left(c_1,\ldots, c_l,\parity\right)$. Our main objects of study will be modules of the form
\begin{align}
	\bigwedge\nolimits^{\ur,\ud} V_r:= \bigwedge\nolimits^{n_1,c_1} V_r \otimes\cdots\otimes \bigwedge\nolimits^{n_l,c_l} V_r\otimes \bigwedge\nolimits^{r,\parity} V_r.
	\end{align}

\noindent This has a monomial basis indexed by the set
\begin{align}
	\Xi_{\ur,\ud}:=\Xi_{r;n_1,c_1}\times \cdots\times\Xi_{r;n_l,c_l}\times\Xi_{r;r,\parity},
	\end{align}

\noindent where for $\lambda=\left(\lambda^i\right)_{1\leq i\leq l+1}\in\Xi_{\ur,\ud}$ we set
\begin{align}
	v_{\lambda}:= v_{\lambda^1}\otimes \cdots \otimes v_{\lambda^{l+1}}.
	\end{align}
	
\noindent It will sometimes be convenient to think of $\lambda$ as a 01-matrix $\lambda=(\lambda^i_j)_{1\leq i\leq l+1,j\in I_r^+}$ whose $i^\text{th}$ row is $\lambda^i$.

For $r<\infty$ and $\lambda\in\Xi_{\ur,\ud}$, $v_\lambda$ has weight
\begin{align}
	| \lambda | = | \lambda^1 | + \cdots + | \lambda^l |\in P_r
	\end{align}
	
\noindent Define a poset structure on $\Xi_{\ur,\ud}$ by declaring that if $\lambda,\mu\in \Xi_{\ur,\ud}$ then $\lambda\leq \mu$ iff $| \lambda | = | \mu |$ and $| \lambda^1 | + \cdots + | \lambda^i | \geq | \mu^1 | + \cdots + | \mu^i |$ for each $1\leq i\leq l+1$.

For $r<s<\infty$ there are maps $\bigwedge\nolimits^{\ur,\ud}V_r\into\bigwedge\nolimits^{\us,\ud}V_s$ induced by inclusion on the first $l$ factors and (\ref{wedgemap}) on the last. There is a corresponding map $\Xi_{\ur,\ud}\into\Xi_{\us,\ud}$ given by setting `new' $\lambda^i_{j}$ equal to $c_i$ if $1\leq i\leq l$ and by (\ref{indexsetwedgemap}) on the last factor. So we have
\begin{align}\label{directlimit}
	\bigwedge\nolimits^{\uinfty,\ud}V= \lim_{\longrightarrow} \, \bigwedge\nolimits^{\ur,\ud} V_r, \qquad \qquad \Xi_{\uinfty,\ud}=\lim_{\longrightarrow} \, \Xi_{\ur,\ud}.
	\end{align}

\noindent We will freely identify elements of $\Xi_{\ur,\ud}$ with their image in $\Xi_{\uinfty,\ud}$ and write ${\Xi_{\uinfty,\ud}=\bigcup_{r=1}^\infty \Xi_{\ur,\ud}}$. Similarly we will consider the $\bigwedge\nolimits^{\ur,\ud} V_r$ as $\mathfrak{sl}_{I_r}$-submodules of $\bigwedge\nolimits^{\uinfty,\ud} V$.

The maps (\ref{directlimit}) are order preserving so there is an induced partial order on $\Xi_{\uinfty,\ud}$.

\begin{wts}\label{wts}
	The maps (\ref{directlimit}) are order preserving but they are \emph{not} weight preserving. If $\lambda\in \Xi_{\ur,\ud}\subseteq \Xi_{\uinfty,\ud}$ we will write $\lvert \lambda \rvert_s~(s\geq r)$ to denote the weight of $v_\lambda$ when considering $\lambda$ as an element of $\Xi_{\us,\ud}$.
	\end{wts}

\begin{rmk:wedgenotationcomp}
	Most of our notations match those of \cite[\S2.2]{BLW13}. The only notable differences are that our $V_r$ and $W_r$ correspond to $V_{I_r}$ and $W_{I_r}$ in \emph{loc. cit.} and, when $r<\infty$, our $\Xi_{r;n,\parity}$ and $\Xi_{\ur,\ud}$ correspond to their $\Lambda_{I_r;n,\parity}$ and $\Lambda_{I_r;\ur,\ud}$ respectively. They do not consider semi-infinite wedges so don't have notation for these indexing sets when $r=\infty$.
	\end{rmk:wedgenotationcomp}

\noindent The following straightforward proposition is the basis for super-duality.
	
\begin{combinatorialsuperduality}\label{combinatorialsuperduality}
	We have equality of posets $\Xi_{\uinfty,\underline{0}}=\Xi_{\uinfty,\underline{1}}$ and this induces an $\mathfrak{sl}_{\mathbb{Z}}$-module isomorphism 
	\begin{align}
		\begin{split}
			\bigwedge\nolimits^{\uinfty, \underline{0}} 	V	& \longrightarrow\bigwedge\nolimits^{\uinfty, \underline{1}} V \\
			v_{\lambda}							& \longmapsto v_{\lambda}
			\end{split}
		\end{align}
	\end{combinatorialsuperduality}

\section{Tensor product categorifications}\label{section:uniquenessofTPCs}

\subsection{Recollection of definitions}

We recall some important definitions. See \cite[\S2.3-2.6]{BLW13} for more details.

\begin{HeckeDef}
	The (degenerate) \emph{affine Hecke algebra} $AH_k$ is the $\mathbb{C}$-algebra with generators $x_1,\ldots, x_k, t_1,\ldots t_{k-1}$ such that $x_1,\ldots, x_k$ generate a polynomial algebra $\mathbb{C}\left[ x_1,\ldots,x_k\right]$, $t_1,\ldots ,t_{k-1}$ generate a copy of the symmetric group $S_k$ with $t_j$ corresponding to the transposition $(j~j+1)$, and
	\begin{align}
		t_ix_j-x_{t_i(j)}t_i=
			\begin{cases}
				1	& \text{if } j=i+1\\
				-1	& \text{if }j=i \\
				0	& \text{otherwise}
				\end{cases}
		\end{align}
	
	\noindent For $\omega\in P_r^+$, the \emph{cyclotomic affine Hecke algebra} $AH_k^{\omega}$ is the quotient of $AH_k$ by the two-sided ideal generated by $\prod_{i\in I_r}(x_1-i)^{\omega\cdot\alpha_i}$.
	\end{HeckeDef}

The image of $\mathbb{C}\left[ x_1,\ldots, x_k\right]$ in $AH_k^{\varpi}$ is a finite-dimensional commutative algebra, so it contains mutually orthogonal idempotents $\left\{ \, 1_{\boldsymbol{i}} \, \middle| \, \boldsymbol{i}\in\mathbb{C}^k \, \right\}$ such that $1_{\boldsymbol{i}}$ acts on any finite-dimensional $AH_k^{\varpi}$-module M as projection onto
\begin{align}
	M_{\boldsymbol{i}} = \left\{ \, v\in M \, \middle| \, (x_j-i_j)^Nv=0 \text{ for } 1\leq j\leq k \text{ and } N\gg 0 \, \right\}.
	\end{align}

\noindent Define 
\begin{align}
	AH_{r,k}^{\varpi}:=\bigoplus_{\boldsymbol{i},\boldsymbol{j}\in I_r^k} 1_{\boldsymbol{i}}AH_k^{\varpi} 1_{\boldsymbol{j}}=AH_k^{\varpi}\left.\middle/\right.\langle \, 1_{\boldsymbol{i}} \, | \, \boldsymbol{i}\notin I_r^k \, \rangle.
	\end{align}

\begin{QHA}
	For consistency we will phrase all categorical actions in this paper in terms of (cyclotomic) affine Hecke algebras, but they could equally be phrased using appropriate (cyclotomic) quiver Hecke algebras (also known as KLR algebras \cite{KL09}, \cite{Rou08}) as described in \cite[Sections 2.3 and 2.4]{BLW13}.
	\end{QHA}

\begin{CategoricalAction}
	Let $\mathcal{C}$ be a Schurian category with an endofunctor $F\in\End\left(\mathcal{C}\right)$, a right adjoint $E$ to $F$ (with a fixed adjunction), and natural transformations $x\in\End\left(F\right)$ and $t\in\End\left(F^2\right)$. Via the adjunction there are induced natural transformations $x\in\End\left(E\right)$ and $t\in\End\left(E^2\right)$. For $i\in I_r$ let $F_i$ (resp.\;$E_i$) be the endofunctor of $\mathcal{C}$ defined by setting $F_iM$ (resp.\;$E_iM$) equal to the generalized $i$-eigenspace of $x$ on $FM$ (resp.\;$EM$). We say this data defines a \emph{categorical $\mathfrak{sl}_{I_r}$-action} on $\mathcal{C}$ if:
	\begin{enumerate}[leftmargin=*, align=left, label=(CA\arabic{*}), ref=(CA\arabic{*})]
		\item \label{CA1} $F$ decomposes as a direct sum $F=\bigoplus_{i\in I_r}F_i$;
		\item \label{CA2} The endomorphisms $x_j:=F^{k-j}xF^{j-1}$ and $t_j:=F^{k-j-1}tF^{j-1}$ of $F^k$ satisfy the relations of the affine Hecke algebra $AH_k$ for all $k\geq 0$;
		\item \label{CA3} $F$ is isomorphic to a right adjoint of $E$;
		\item \label{CA4} The endomorphisms $f_i$ and $e_i$ of $\left[ \mathcal{C}\right]$ induced by $E_i$ and $F_i$ make it into an integrable $\mathfrak{sl}_{I_r}$-module such that the classes of indecomposable projectives are weight vectors.
		\end{enumerate}
	\end{CategoricalAction}

As a consequence of these axioms, $E=\bigoplus_{i\in I_r} E_i$ and $E_i$ and $F_i$ are biadjoint for all $i\in I_r$. The data of a categorical $\mf{sl}_{I_r}$-action on $\mc{C}$ implies an action of the associated Kac-Moody 2-category as defined by \cite{Rou08}, \cite{KL10}, and \cite{CL16} (these definitions were shown to be equivalent in \cite{Bru16}).

Recall the definition of a (Schurian) highest weight category $\mathcal{C}$ in the sense of \cite{CPS88}. As in \cite[Definition 2.8]{BLW13} we allow the associated poset $(\Xi, \leq)$ to be infinite as long as it is interval-finite: for $\lambda,\mu\in\Xi$ the set $\left\{ \, \nu\in\Xi \, \middle| \,  \lambda\leq\nu\leq\mu \, \right\}$ is finite. We write $L$, $\Delta$, and $P$ for the irreducibles, standards, and projectives in $\mathcal{C}$ respectively. Let $\mathcal{C}^{\Delta}$ be the full subcategory of $\mathcal{C}$ consisting of objects with a $\Delta$-flag and let $\left[\mc{C}^{\Delta}\right]$ denote its complexified Grothendieck group.

\begin{TPCdef}
	Take $r\in\NN\cup\{\infty\}$ and $\parity\in\{0,1\}$. An \emph{$\mathfrak{sl}_{I_r}$-tensor product categorification} (TPC) of type $(\ur,\ud)$ is a Schurian highest weight category with poset $\Xi_{\ur,\ud}$ and a categorical $\mathfrak{sl}_{I_r}$-action such that
	\begin{enumerate}[leftmargin=*, align=left, label=(TPC\arabic{*}), ref=(TPC\arabic{*})]
		\item \label{TPC1} $F_i$ and $E_i$ restrict to endofunctors of $\mathcal{C}^{\Delta}$;
		\item \label{TPC2} There is a linear isomorphism 
			\begin{align}
				\begin{split}
					\left[ \mathcal{C}^{\Delta}\right]		& \longrightarrow \bigwedge\nolimits^{\ur,\ud} V_r, \\
					\left[ \Delta(\lambda)\right]			& \longmapsto v_\lambda 
					\end{split}
				\end{align}
		
			intertwining the endomorphisms induced by $E_i$ and $F_i$ and the Chevalley generators $e_i,f_i\in\mathfrak{sl}_{I_r}$. 
		\end{enumerate}
	\end{TPCdef}

\noindent As $\left[\mc{C}\right]\into\left[\mc{C}^{\Delta}\right]$, the axiom \ref{CA4} is actually a consequence of \ref{TPC2}.

The following proposition immediately follows from the definitions and Proposition \ref{combinatorialsuperduality}.

\begin{TPCsuperduality}\label{prop:TPCsuperduality}
	Every $\mathfrak{sl}_{\mathbb{Z}}$-TPC of type $(\uinfty,\underline{0})$ is also a TPC of type $(\uinfty,\underline{1})$ and vice-versa.
	\end{TPCsuperduality}

\subsection{Strategy for uniqueness} \label{section:uniqueintro}

The rest of this section will be dedicated to proving the following theorem (see \cite{LW13} or \cite{BLW13} for the definition of a strongly equivariant equivalence):

\begin{uniquenessofTPCs}\label{thm:uniquenessofTPCs}
	Suppose that $\mathcal{C}$ and $\mathcal{C'}$ are $\mathfrak{sl}_{I_r}$-TPCs of type $(\ur,\ud)$. Then there is a strongly equivariant equivalence $\mathbb{G}:\mathcal{C}\to\mathcal{C'}$ with $\mathbb{G}L(\lambda)=L'(\lambda)$ for all $\lambda\in\Xi_{\ur,\ud}$. 
	\end{uniquenessofTPCs}
	
For $r<\infty$ this is a special case of the main result in \cite{LW13} so we will only actually provide a proof for the case $r=\infty$.

To establish the theorem, it suffices to show that for a given type $(\ur,\ud)$ there exists a Schurian category $\mathcal{D}$ with a categorical $\mathfrak{sl}_{I_r}$-action such that for any $\mathfrak{sl}_{I_r}$-TPC $\mathcal{C}$ of type $(\ur,\ud)$ there is an exact functor $\mathbb{U}:\mathcal{C} \to \mathcal{D}$ satisfying the following:

\begin{enumerate}[label=(U\arabic{*}), ref=(U\arabic{*})]
	\item \label{U1} $\mathbb{U}$ is strongly equivariant;
	
	\item \label{U2} $\mathbb{U}$ is fully faithful on projectives;
	
	\item \label{U3} for each $\lambda\in \Xi_{\ur,\ud}$, $Y(\lambda):= \mathbb{U} P(\lambda)$ is independent (up to isomorphism) of the choice of $\mathcal{C}$.
	\end{enumerate}

\noindent See \cite[\S2.7]{BLW13} for an explanation of why this is sufficient.

For the rest of Section \ref{section:uniquenessofTPCs} we fix an $\mathfrak{sl}_{\mathbb{Z}}$-TPC $\mc{C}$ of type $(\uinfty, \ud)$. In \S \ref{subsection:truncation} below we define a subquotient $\mathcal{C}_r$ of $\mathcal{C}$ that is an $\mathfrak{sl}_{I_r}$-TPC of type $(\ur,\ud)$ and construct the corresponding functors $\mathbb{U}_r$ satisfying \ref{U1}-\ref{U3}. Then in \S \ref{subsection:stablemods} we construct a functor $\mathbb{U}$ from $\mathcal{C}$ to a category of `stable modules' using the functors $\mathbb{U}_r$ and show that $\mathbb{U}$ satisfies \ref{U1}-\ref{U3}, thus establishing the theorem. Our proof of Theorem \ref{thm:uniquenessofTPCs} for $r=\infty$ closely follows the argument in \cite[Section 4]{BLW13}. The reader is referred there for many of the proofs in this section as they pass over with minimal change. Our $\Xi_{\ur,\ud}$ corresponds to $\Lambda_r$ in their notation and we have $I=\ZZ$.

\subsection{Reduction to finite intervals}\label{subsection:truncation}

Let $r_0:=\max\left\{n_1,\ldots, n_l\right\}$ and take $r$ such that $r_0\leq r<\infty$. Define subsets $\Xi^{\leq}_{\ur,\ud}$ and $\Xi^<_{\ur,\ud}$ of $\Xi_{\uinfty,\ud}$ as follows. Take $\lambda\in \Xi_{\uinfty,\ud}$ and choose $s\geq r$ such that $\lambda\in \Xi_{\us,\ud}$. Then $\lambda\in \Xi^\leq_{\ur,\ud}$ if it satisfies the following conditions:
\begin{empheq}[left=\empheqlbrace]{align}
	\begin{split}
	\sum_{i=1}^k \sum_{\substack{j\in I_s \\ j\leq h \\ \lambda^i_{j}\neq c_i}} (-1)^{c_i} \geq 0	& \text{ for all } h< \min (I_r^+) \text{ and } 1\leq k\leq l+1 \\
	\sum_{i=1}^k \sum_{\substack{j\in I_s \\ j\geq h \\ \lambda^i_{j}\neq c_i}} (-1)^{c_i} \leq 0		& \text{ for all } h> \max (I_r^+) \text{ and } 1\leq k\leq l+1
	\end{split}
	\end{empheq}

\noindent (we write $\parity$ as $c_{l+1}$ for convenience) and $\lambda\in \Xi^<_{\ur,\ud}$ if at least one of the above inequalities is strict. This definition is independent of the choice of $s$. These are ideals of $\Xi_{\uinfty,\ud}$ and $\Xi_{\ur,\ud}=\Xi^{\leq}_{\ur,\ud}\setminus\Xi^<_{\ur,\ud}$. Moreover the vector subspaces of $\bigwedge\nolimits^{\uinfty, \ud} V$ spanned by the corresponding $v_{\lambda}$ are $\mathfrak{sl}_{I_r}$-submodules.

Let $\mathcal{C}_{\leq r}$ (resp. $\mathcal{C}_{<r}$) be the Serre subcategory of $\mathcal{C}$ generated by those $L(\lambda)$ with $\lambda\in\Xi^{\leq}_{\ur,\ud}$ (resp. $\Xi^<_{\ur,\ud}$) and let $\mathcal{C}_r$ be the Serre quotient category
\begin{align}
	\mathcal{C}_r=\mathcal{C}_{\leq r}\left.\middle/\right.\mathcal{C}_{<r}.
	\end{align}
	
\noindent All three of these are highest weight categories. The functors $F_i$ and $E_i$ with $i\in I_r$ preserve the subcategories $\mathcal{C}_{\leq r}$ and $\mathcal{C}_{<r}$ and thus they have induced actions on $\mathcal{C}_r$. The following proposition is easily established by checking the necessary axioms:

\begin{trunc}
	The category $\mathcal{C}_r$ is an $\mathfrak{sl}_{I_r}$-tensor product categorification of type $(\ur,\ud)$ under the induced action of $F_{I_r}:=\bigoplus_{i\in I_r} F_i$, together with the restrictions of $x$ and $t$ to $F_{I_r}$ and $F_{I_r}^2$ respectively.
	\end{trunc}

Define an equivalence relation `$\sim$' on $\Xi_{\uinfty,\ud}$ as follows. Take $\lambda,\mu \in\Xi_{\uinfty,\ud}$. Take $r<\infty$ such that $\lambda,\mu\in\Xi_{\ur,\ud}$ and write $\lambda\sim \mu$ if $\lvert \lambda\rvert_r=\lvert\mu\rvert_r$ (see Remark~\ref{wts}). This definition is independent of the choice of $r$. For $\lambda\in \Xi_{\uinfty,\ud}$, let $\mathcal{C}_{[\lambda]}$ be the Serre subcategory of $\mathcal{C}$ generated by those $L(\mu)$ with  $\mu\sim \lambda$. Then we can decompose
\begin{align} \label{eq:blocks}
	\mathcal{C}=\bigoplus_{[\lambda]\in \Xi_{\uinfty,\ud}/\sim} \mathcal{C}_{[\lambda]}
	\end{align}

\noindent If $r_0\leq r <\infty$ then $\Xi_{\ur,\ud}$ has a unique maximal element $\kappa_r$. It is the 01-matrix with the 1s in each row as far left as possible. The corresponding vector $v_{\kappa_r}$ is the highest weight vector in $\bigwedge^{\ur, \ud} V_r$.

\begin{prinj} \label{Lem2.20}
	$L(\kappa_r)\in\mathcal{C}$ is both projective and injective.
	\end{prinj}
	
\begin{proof}
	This follows from the proof of \cite[Lemma~2.20]{BLW13}, using (\ref{eq:blocks}) above in place of \cite[(2.16)]{BLW13}.
	\end{proof}
		
\noindent For $r_0\leq r<\infty$, define $T_r:=\bigoplus_{k\geq 0}F^k_{I_r}L(\kappa_r)$ and $H_r:=\bigoplus_{k\geq 0} AH^{\lvert \kappa_r\rvert}_{r,k}$.

\begin{hecketensor}\cite[Theorem~4.1]{BLW13}
	The action of the affine Hecke algebras on $T_r$ induces a canonical algebra isomorphism $H_r\cong \End_\mathcal{C}(T_r)$.
	\end{hecketensor}

\noindent In particular, considering $T_r$ as a left $H_r$-module we can define a functor:
\begin{align}
	\mathbb{U}_r:=\Hom_\mathcal{C}(T_r,-):\mathcal{C}\longrightarrow \text{mod-}H_r
	\end{align}
	
\noindent By Lemma \ref{Lem2.20}, $T_r$ is both projective and injective and so $\mathbb{U}_r$ is exact. In the finite interval case, mod-$H_r$ fills the role of the category $\mathcal{D}$ mentioned in Section \ref{section:uniqueintro}; that is, there exists a categorical $\mathfrak{sl}_{I_r}$-action on mod-$H_r$ such that the functor $\mathcal{C}_r\to \text{mod-}H_r$ induced by $\mathbb{U}_r$ satisfies (U1)-(U3) (see \cite[Theorem~2.14 and Lemma~2.16]{BLW13}).

\subsection{Stable modules}\label{subsection:stablemods}

Take $r_0\leq r<\infty$. Then there exist $a\geq 0$, $p_1,\ldots,p_a\in \mathbb{N}$, and $s_1,\ldots,s_a\in I_r$ such that
\begin{align} \label{eq:hwvs}
	f^{(p_1)}_{s_1} f^{(p_2)}_{s_2} \cdots f^{(p_a)}_{s_a} v_{\kappa_{r+1}}=v_{\kappa_r}
	\end{align}

\noindent in $\bigwedge^{\uinfty, \ud} V$. Let $p=p_1+\cdots +p_a$. For $i\in \mathbb{Z}$ and $m\geq 1$, let $F^{(m)}_i$ denote the summand of $F^m_i$ that induces $f^{(m)}_i$ on the level of the Grothendieck group.

\begin{Lem4.2}\label{Lem4.2}
	There is an algebra embedding $\phi_r:H_r\to H_{r+1}$, independent of the choice of $\mathcal{C}$, such that $e_r:=\phi_r\left(1_{H_r}\right)$ acts as projection
	\begin{align}
		F^p_{I_{r+1}} L(\kappa_{r+1})\to F^{(p_1)}_{s_1} F^{(p_2)}_{s_2} \cdots F^{(p_a)}_{s_a} L(\kappa_{r+1})\cong L(\kappa_r)
		\end{align}
	
	\noindent in $\mc{C}$. Moreover there is an isomorphism $\theta_r:T_r\isoto e_r T_{r+1}$ intertwining the action of $H_r$ on $T_r$ with its action on $T_{r+1}$ via $\phi_r:H_r\isoto e_r H_re_r\subseteq H_{r+1}$.
	\end{Lem4.2}

\begin{proof}
	The assumption $r\geq r_0$ ensures that an identity of the form (\ref{eq:hwvs}) holds. With this, the proof is the same as that of \cite[Lemma 4.2]{BLW13}.
	\end{proof}

\begin{Lem4.2rmk}
	It is more natural to think of the above map in terms of the diagrammatics of the quiver Hecke algebra as in \cite[Section 4.1]{BLW13}. With this perspective, the map between cyclotomic quotients is induced by tensoring on the left with the diagram that gives projection $F^p_{I_{r+1}} \to F^{(p_1)}_{s_1} F^{(p_2)}_{s_2} \cdots F^{(p_a)}_{s_a}$ in the categorification of the half quantum group $\mc{U}_q^-\mf{sl}_{I_{r+1}}$. This diagram can be given explicitly, see for example \cite[Lemma~4.1]{Rou08}.
	\end{Lem4.2rmk}

Since $\phi_r:H_r\isoto e_rH_{r+1}e_r$, we can induct and restrict modules between the $H_r$ for varying $r$. Given a right (resp. left) $H_{r+1}$-module $M$, we consider $Me_r$ (resp. $e_rM$) as a right (resp. left) $H_r$-module via $\phi_r$. For $M\in \text{mod-}H_r$ and $N\in \text{mod-}H_{r+1}$, define 
\begin{align}
	\begin{gathered}
		\Ind_r^{r+1} M:=M\otimes_{H_r} e_r H_{r+1}\in \text{mod-}H_{r+1}\\
		\Res^{r+1}_r N:=Ne_r\in \text{mod-}H_r
		\end{gathered}
	\end{align}

\noindent More generally, we define functors
\begin{align}
	\begin{split}
		\Ind_r^s	& :=\Ind_{s-1}^s \circ\cdots\circ \Ind_r^{r+1} \\
		\Res_r^s	& :=\Res_r^{r+1}\circ\cdots\circ \Res_{s-1}^s
		\end{split}
	\end{align}
	
\noindent for any $s>r$.

\begin{hinfty}
	Let mod-$H_{\infty}$ be the category with objects sequences ${M=(M_r, \iota_r)_{r\geq r_0}}$ such that $M_r\in \text{mod-}H_r$ and
	\begin{equation}
		\iota_r:M_r\longrightarrow \Res^{r+1}_r M_{r+1}\subseteq M_{r+1}
		\end{equation}
	
	\noindent is an $H_r$-module isomorphism for each $r$. A morphism $f:M\to N$ in mod-$H_{\infty}$ is a sequence $f=(f_r)_{r\geq r_0}$ with $f_r\in \Hom_{H_r}(M_r,N_r)$ such that the following diagram commutes for all $r\geq r_0$:
	\begin{equation}
		\begin{tikzcd}
			M_r \ar[d, "f_r"'] \ar[r, "\iota_r"]	& M_{r+1} \ar[d, "f_{r+1}"] \\
			N_r \ar[r, "\iota_r"]			& N_{r+1} 
			\end{tikzcd}
		\end{equation}
	\end{hinfty}
	
For $r\geq r_0$, define a functor $\text{st}_r: \text{mod-}H_r\to \text{mod-}H_\infty$ by setting $\text{st}_r(M)=(M_s,\iota_s)_{s\geq r_0}$ where
\begin{align}
	M_s:=\begin{cases}
			\Ind_r^s M		& \text{if } s>r \\
			M			& \text{if } s=r \\
			\Res^r_s M	& \text{if } r_0\leq s<r.
			\end{cases}
	\end{align}

\noindent For $r_0\leq s<r$, $M_s=\text{Res}^{s+1}_s M_{s+1}$ so we can take $\iota_s=1_{M_s}$, and for $s\geq r$, $\iota_s$ is defined via the obvious natural isomorphism $1\isoto \Res^{s+1}_s \Ind^{s+1}_s$.

The functor $\text{st}_r$ corresponds to $\text{pr}_r^!$ in \cite[\S4.2]{BLW13}.

\begin{stable}
	An object in mod-$H_\infty$ is \emph{$r$-stable} if it is in the essential image of $\text{st}_r$. It is \emph{stable} if it is $r$-stable for some $r\geq r_0$. Let mod-$H$ denote the full subcategory of mod-$H_\infty$ consisting of stable modules.
	\end{stable}

Recall the functors $\mathbb{U}_r$ from \S \ref{subsection:truncation}. Define a functor
\begin{align}
	\begin{split}
		\mathbb{U}:\mathcal{C}	& \longrightarrow \text{mod-}H_\infty \\
		M					& \longmapsto(\mathbb{U}_r M,\iota_r)_{r\geq r_0}
		\end{split}
	\end{align}

\noindent where
\begin{align}
	\begin{split}
		\iota_r:\Hom_{H_r}(T_r,M)		& \longisoto \Hom_{H_{r+1}}(e_r T_{r+1},M) \\
		\phi						& \longmapsto \phi\circ (\theta_r)_{-1}
		\end{split}
	\end{align}

\noindent It is defined on morphisms by setting $\mathbb{U}f=(\mathbb{U}_rf)_{r\geq r_0}$.

\begin{Uprop} \label{thm:Uprop}
	The category mod-$H$ is Schurian and $\mathbb{U}M\in \text{mod-}H$ for every $M\in\mathcal{C}$. Moreover, there exists a categorical $\mathfrak{sl}_\mathbb{Z}$-action on mod-$H$ such that $\mathbb{U}$ satisfies (U1)-(U3).
	\end{Uprop}
	
\begin{proof}
	The analogous results in \cite{BLW13} comprise Theorems 4.7, 4.9, and 4.10, and \S4.3. The proof is formally identical: our Lemma \ref{Lem2.20} takes the place of \cite[Lemma~2.20]{BLW13} , our Lemma \ref{Lem4.2} takes the place of \cite[Lemma~4.2]{BLW13}, and \cite[Theorem~2.24]{BLW13} and its proof go through unchanged.
	\end{proof}

\noindent By the discussion in \S \ref{section:uniqueintro} we have established Theorem \ref{thm:uniquenessofTPCs}.

\begin{rmk:prinjectives}
	As mentioned in the proof above, \cite[Theorem~2.24]{BLW13} holds in our setting. This gives a classification of projective-injective objects in $\mc{C}$. Together with Theorem~\ref{infranktpc} this yields a classification on projective-injective modules in the infinite-rank categories $\catO^{++}_{\parity}$ described in the next section.
	\end{rmk:prinjectives}

\section{Categories $\mathcal{O}$}\label{section:catO}

Fix $\parity\in\{0,1\}$. The aim of this section is to show that a certain infinite-rank limit $\mathcal{O}_{\parity}^{++}$ of parabolic BGG categories of general linear Lie superalgebras is an $\mf{sl}_{\ZZ}$-tensor product categorification of type $(\uinfty, \ud)$. In \S\ref{subsec:Set-up} we set up the necessary notation and define the category $\mathcal{O}^{++}_{\parity}$. In \S\ref{subsec:HWC} we show that it has enough projectives and conclude that it is a highest weight category. Finally in \S\ref{subsec:cataction} we describe the actions of $E$ and $F$ on $\mathcal{O}^{++}_{\parity}$ and show that it is an $\mf{sl}_{\ZZ}$-TPC. This yields the super duality equivalence $\catO^{++}_0\cong\catO^{++}_1$.

Since $\parity$ will be fixed until the discussion of super duality at the end of this section, we will write $\catO^{++}$ for $\catO^{++}_{\parity}$ up to that point. Our constructions also depend on the sequences $n_1,\ldots, n_l\in\mathbb{N}$ and $c_1,\ldots,c_l\in\left\{0,1\right\}$ which we fixed in Section \ref{sec:intro}. For cleanness we will always drop reference to these from our notation.
	
\subsection{Set up}\label{subsec:Set-up}

Let $m=\sum_{c_i=1} n_i$ and $n=\sum_{c_i=0} n_i$. For notational convenience we will sometimes write $n_{l+1}=r$ and $c_{l+1}=\parity$.

For $1\leq j\leq m+n+r$, define $p_j=c_i\in \{0,1\}$, where $1\leq i\leq l+1$ is maximal such that ${n_1+\cdots +n_{i-1}<j}$. For $r<\infty$ let $U_r$ be the vector superspace with basis $u_1,\ldots ,u_{m+n+r}$, where $u_j$ has degree $p_j$. Let $u_1^*\ldots,u_{m+n+r}^*\in U_r^*$ be the dual basis. Let $\mathfrak{g}_r:=\mathfrak{gl}(U_r)$ be the Lie superalgebra of endomorphisms of $U_r$ under the supercommutator bracket. Let $U=U_\infty=\bigcup_{r=1}^\infty U_r$ and let $\mathfrak{g}=\mathfrak{g}_\infty= \displaystyle{\lim_{\longrightarrow}} \, \mathfrak{g}_r$ be the Lie superalgebra of endomorphisms of $U$ that vanish on all but finitely many of the $u_j$. 

For $r<\infty$, let $\left\{ \, e_{ij} \,\middle|\, 1\leq i,j\leq m+n+r \, \right\}$ be the basis of matrix units for $\mathfrak{g}_r$. Denote by $\mathfrak{b}_r$  the Borel subalgebra of $\mathfrak{g}_r$ consisting of upper triangular matrices. Let $\mathfrak{h}_r$ be the Cartan subalgebra of $\mathfrak{g}_r$ with basis $\left\{ \, e_{ii} \,\middle|\, 1\leq i\leq m+n+r \, \right\}$ and let $\left\{ \delta_i \,\middle|\, 1\leq i\leq m+n+r\right\}\subseteq\mathfrak{h}_r^*$ be the dual basis. Define a bilinear form $\left(-,-\right)$ on $\mathfrak{h}_r^*$ by declaring that
\begin{equation}
	\left( \delta_i,\delta_j \right)=\begin{cases}
							(-1)^{p_i}	& \text{if } i=j \\
							0		& \text{otherwise}
							\end{cases}
	\end{equation}

\noindent Let $\Phi_r=\left\{ \, \delta_i - \delta_j \,\middle|\, 1\leq i,j\leq m+n+r, \, i\neq j \, \right\}$ be the root system for $\mathfrak{g}_r$. A root $\delta_i-\delta_j$ is even if $p_i=p_j$ and is odd otherwise. It is positive (with respect to $\mathfrak{b}_r$) if $i<j$ and is negative otherwise. Let $\Phi_{r, \overline{0}}$, $\Phi_{r, \overline{1}}$, $\Phi_{r}^+$, and $\Phi_r^-$ be the sets of even, odd, positive, and negative roots respectively and decompose $\Phi_r^+=\Phi_{r, \overline{0}}^+\sqcup \Phi_{r, \overline{1}}^+$ and $\Phi_r^-=\Phi_{r, \overline{0}}^-\sqcup \Phi_{r, \overline{1}}^-$ in the obvious way.

For $r<\infty$, define the Weyl vector
\begin{align}
	\overline{\rho}_r = \frac{1}{2} \sum_{\alpha\in \Phi_{r, \overline{0}}^+} \alpha - \frac{1}{2} \sum_{\beta\in\Phi_{r, \overline{1}}^+} \beta \in\mathfrak{h}_r^*
	\end{align}

\noindent and the normalized version
\begin{align}
	\rho_r=\overline{\rho}_r+\left( \frac{n-m+1-(-1)^{\parity} r}{2}\right) \sum_{i=1}^{m+n+r} (-1)^{p_i} \delta_i.
	\end{align}

\noindent This has the following properties:
\begin{align} \label{weylprop}
		\left( \rho_r,\delta_i-\delta_{i+1} \right)	 &= \begin{cases}
											(-1)^{p_i}	& \text{if } p_i=p_{i+1} \\
											0		& \text{otherwise}
											\end{cases}
		&&
		\left( \rho_r,\delta_{m+n+r} \right)		 &=\begin{cases}
											1-r	& \text{if } \parity=0 \\
											r	& \text{if } \parity=1
											\end{cases}
	\end{align}	

\noindent and if $r<s$ then $\rho_r$ is the restriction of $\rho_s$ to $\mathfrak{h}_r$.

For $r<\infty$, let 
\begin{equation}
	X_r=\bigoplus_{j=1}^{m+n+r} \mathbb{Z}\delta_j
	\end{equation}
	
\noindent be the integral weight lattice for $\mathfrak{g}_r$ and define
\begin{equation*}
	\begin{gathered}
			X_r^+ 	 = \left\{ \, \lambda\in X_r \,\middle|\, (-1)^{p_j}(\lambda+\rho_r,\delta_j-\delta_{j+1})>0 \text{ unless } j=n_1+\cdots+n_i \text{ for some }i \right\} \vspace{.1in},\\
			X_r^{++} 	 = \left\{ \lambda\in X_r^+ \,\middle|\, (-1)^{\parity}(\lambda,\delta_{m+n+r})\geq 0\right\}.
		\end{gathered}
	\end{equation*}

\noindent Using the natural embeddings $X_r^{++}\subseteq X_{r+1}^{++}$, define $X^{++}=X_\infty^{++}=\bigcup \, X_r^{++}$.

We will identify these sets with certain indexing sets for wedges from Section \ref{section:combinatorics}. Take $r<\infty$. Take $\lambda\in X_r$ and define a corresponding 01-matrix $( \lambda^i_{j} )_{1\leq i\leq l+1, j\in \mathbb{Z}}$ as follows. Fix $1\leq i\leq l+1$ and let $k=n_1+\cdots +n_{i-1}$. Set
\begin{align}
	\lambda^i_{j} = \begin{cases}
					1-c_i	& \text{for } j= \left( \lambda+\rho_r,\delta_{k+1} \right),\ldots ,\left( \lambda+\rho_r,\delta_{k+n_i} \right) \\
					c_i	& \text{otherwise}
					\end{cases}
	\end{align}

\noindent This provides an identification between $X_r^+$ and the indexing set for the basis of the module
\begin{equation}
	\bigwedge\nolimits^{n_1,c_1}V\otimes\cdots\bigwedge\nolimits^{n_l,c_l}V\otimes \bigwedge\nolimits^{r,\parity}V.
	\end{equation}
	
\noindent We will freely identify these sets and use the inherited partial order on $X_r^+$. For $r<s$, the linear map 
\begin{equation}
	\text{span}\left\{ \, v_\lambda \,\middle|\, \lambda\in X_r^{++} \, \right\} \longrightarrow \text{span}\left\{ \, v_\lambda \,\middle|\, \lambda\in X_s^{++} \, \right\}
	\end{equation}
	
\noindent induced by the inclusion $X_r^{++}\subseteq X_s^{++}$ coincides with the map induced by the assignment (\ref{wedgemap}) and the direct limit along these maps is $\bigwedge\nolimits^{\uinfty, \ud} V$. Thus we can identify $X^{++}$ and $\Xi_{\uinfty, \ud}$. This induces a partial order on $X^{++}$ compatible with the partial orders on the $X_r^{++}$.

\begin{catOfinitewedges}
	For finite $r$, the indexing set $\Xi_{\ur,\ud}$ corresponds to \textbf{a proper subset} of $X_r^{++}$. The former indexes tensor products of wedges of $V_r$s and $W_r$s and the latter indexes tensor products of wedges of $V=V_{\infty}$s and $W=W_{\infty}$s with the wedges in the final tensor factor restricted just enough that the assignment (\ref{wedgemap}) is well behaved.
	\end{catOfinitewedges}

For $r\leq\infty$, define a Levi subalgebra $\mathfrak{l}_r$ of $\mathfrak{g}_r$ by
\begin{align}
	\mathfrak{l}_r=\mathfrak{gl}_{n_1}\oplus \cdots \mathfrak{gl}_{n_l}\oplus \mathfrak{gl}_r
	\end{align}

\noindent and let $\mathfrak{p}_r=\mathfrak{l}_r+\mathfrak{b}_r$ be the corresponding parabolic subalgebra. For $\lambda\in X_r^+$ (or $\lambda\in X^{++}$ if $r=\infty$), let $L_r^0(\lambda)$ be the irreducible $\mathfrak{l}_r$-module of highest weight $\lambda$. The corresponding parabolic Verma module is
\begin{align}
	\Delta_r(\lambda)=\mathcal{U}\mathfrak{g}_r\otimes_{\mathcal{U}\mathfrak{p}_r} L_r^0(\lambda)
	\end{align}

\noindent It has a unique irreducible quotient $L_r(\lambda)$.

For $r<\infty$, let $\mathcal{O}_r$ be the integral BGG category $\mathcal{O}$ for $\mathfrak{g}_r$; the category of finitely-generated, $\mathfrak{h}_r$-semisimple, integral weight $\mathfrak{g}_r$-modules that are locally $\mathfrak{b}_r$-finite. Morphisms are all (not necessarily even) homomorphisms of $\mf{g}_r$-modules. Let $\mathcal{O}_r^+$ be the parabolic subcategory of $\mathcal{O}^+_r$ associated to $\mathfrak{p}_r$; the full subcategory of $\mathcal{O}_r$ consisting of modules that are $\mathfrak{l}_r$-semisimple and locally $\mathfrak{p}_r$-finite. Equivalently, $\mathcal{O}_r^+$ is the Serre subcategory of $\mathcal{O}_r$ generated by the $L_r(\lambda)$ with $\lambda\in X_r^+$. It contains the parabolic Verma modules $\Delta_r(\lambda)$ for $\lambda \in X_r^+$. Let $\mathcal{O}_r^{++}$ be the Serre subcategory of $\mathcal{O}_r^+$ generated by $\left\{ L_r(\lambda) \,\middle|\, \lambda\in X_r^{++} \right\}$. Finally let $\mathcal{O}^{++}=\mathcal{O}_\infty^{++}$ be the category of finitely generated, finite length, $\mathfrak{h}$-semisimple $\mathfrak{g}$-modules that are locally $\mathfrak{p}_r$-finite for each $r<\infty$ and whose composition factors are of the form $L(\lambda)=L_{\infty}(\lambda)$ with $\lambda\in X^{++}$. It is abelian and contains the parabolic Verma modules $\Delta(\lambda)=\Delta_{\infty}(\lambda)$ for $\lambda\in X^{++}$.

\begin{equivcatOs}
	The analogous categories in \cite[Definition~3.1]{BLW13} are constructed slightly differently. They add a parity assumption to the weight spaces that ensures all morphisms are even. The different constructions yield equivalent categories, see e.g. \cite[\S~2.5]{CL09}.
	\end{equivcatOs}

Take $r\leq \infty$. The supertranspose $x^{\text{st}}$ of a matrix $x=(x_{ij})\in\mf{g}_r$ is the matrix with $(i,j)$-entry $(-1)^{p_i(p_i+p_j)}x_{ji}$. The assignment $x\mapsto x^{\text{st}}$ defines an anti-automorphism of $\mf{g}_r$. If $M\in\catO_r$, define its dual by
\begin{equation}
	M^{\vee}=\bigoplus_{\lambda\in X_r}M_{\lambda}^*
	\end{equation}
	
\noindent with $\mf{g}_r$-action given by
\begin{equation}
	(x\cdot f)(m):=(-1)^{\lvert x\rvert\cdot \lvert f\rvert}f(x^{st}\cdot m),
	\end{equation}

\noindent where $x$ and $f$ are homogeneous and $\lvert \, \cdot \, \rvert$ denotes their parity. This defines exact, contravariant, self-equivalences on $\catO_r$, $\catO_r^+$, and $\catO_r^{++}$. The dual Verma module corresponding to $\lambda\in X_r$ is
\begin{equation}
	\nabla_r(\lambda):=\Delta_r(\lambda)^{\vee}.
	\end{equation}

The following important functors were first introduced in \cite[Definition~3.4]{CWZ06}.

\begin{def:trunc}\cite[Definition 3.1]{CW07}\label{def:trunc}
	For $r<s\leq\infty$ and $M=\displaystyle{\bigoplus_{\lambda\in X_s}} M_{\lambda}\in \mathcal{O}_s^{++}$, define
	\begin{equation}
		\text{tr}_r^s(M) = \bigoplus_{\lambda\in X_r} M_\lambda \in \cato_r^{++},
		\end{equation}

	\noindent where we regard $X_r\subseteq X_s$. This defines an exact \emph{truncation functor} $\tr_r^s:\cato_s^{++}\to\cato_r^{++}$.
	\end{def:trunc}
	
\noindent We will sometimes drop the sub/superscripts when they are clear from context.

\begin{tilt} \label{tilt}
	If $r<s\leq \infty$, $\lambda\in X_s^{++}$, and $Y=L,~\Delta$, or $\nabla$, then
	\begin{align}
		\tr_r^s~Y_s(\lambda)=\begin{cases}
							Y_r(\lambda)	& \text{if } \lambda\in X_r^{++} \\
							0			& \text{otherwise.}
							\end{cases}
		\end{align}
	\end{tilt}
	
\begin{proof}
	For $Y=L$ or $\Delta$ this is \cite[Proposition 7.5]{CLW12}. For $Y=\nabla$ this follows from the $Y=\Delta$ case and the observation that duality commutes with truncation.
	\end{proof}
	
Take $M\in \catO^{++}$ and $r'\in\NN$ such that if $\lambda\in X^{++}$ with ${\left[ M:L(\lambda)\right]\neq 0}$ then $\lambda\in X_r^{++}$. For $r\in\NN$, let $M_r:=\tr^{\infty}_r(M)$. The proposition implies that the composition multiplicities of $M_r$ are independent of $r\geq r'$ in the sense that if $\infty\geq s>r\geq r'$ and $\lambda\in X_s^{++}$ then
\begin{equation}\label{eq:multstable}
	\left[ M_s:L_s(\lambda)\right] = \begin{cases}
								\left[ M_r:L_r(\lambda)\right]	& \text{if } \lambda\in X_r^{++} \\
								0						& \text{otherwise.}
								\end{cases}
	\end{equation}

\noindent The following lemma gives a converse to this statement.

\begin{limitinO}\label{limitinO}
	Take $r'\in\NN$. Suppose we have modules $M_r\in\catO_r^{++}$ and injective $\mf{g}_r$-module maps ${f_r:M_r\to\tr^{r+1}_r(M_{r+1})\subseteq M_{r+1}}$ for all $r\geq r'$. Suppose further that the composition multiplicities of the $M_r$ are independent of $r\geq r'$ as above. Then the $f_r$ are isomorphisms and the direct limit $M:=\displaystyle{\lim_{\longrightarrow}}~M_r$ along the maps $f_r$ is a module in $\catO^{++}$.
	\end{limitinO}	

\begin{proof}
	By assumption, $M_r$ and $\tr^{r+1}_r(M_{r+1})$ have the same composition multiplicites when $r\geq r'$. This implies $f_r$ is an isomorphism and $f_r^{-1}\circ \tr^{r+1}_r$ sends a composition series of $M_{r+1}$ to a composition series of $M_r$ with the same ordered sequence of weights. Now since taking direct limits is exact and $L(\lambda)=\displaystyle_{\lim_{\longrightarrow}}~L_r(\lambda)$ for any $\lambda\in X^{++}$ by Proposition~\ref{tilt}, $M$ has a finite composition series with the same ordered sequence of weights as $M_r$ for any $r\geq r'$. This implies $M$ is finitely generated. Truncation to $M_r$ shows that $M$ is $\mf{h}$-semisimple and locally $\mf{p}_r$-finite for any $r$. So $M\in\catO^{++}$.
	\end{proof}

\begin{rmk:deltamultstable}\label{rmk:deltamultstable}
	The same conclusion holds if $M_r\in(\catO^{++}_r)^{\Delta}$ for all $r\geq r'$ and we replace composition multiplicities with $\Delta$-multiplicities. Moreover, in this situation we can conclude that $M\in(\catO^{++})^{\Delta}$.
	\end{rmk:deltamultstable}

\subsection{Highest weight structure}\label{subsec:HWC}

\begin{++hwc}
	If $r<\infty$ then $\mathcal{O}_r^{++}$ is a highest weight category with weight poset $(X_r^{++}, \leq)$, standard objects $\Delta_r(\lambda)$, and costandard objects $ \nabla(\lambda)$.
	\end{++hwc}

\begin{proof}
	The parabolic category $\mathcal{O}_r^+$ is a highest weight category with weight poset $(X_r^+, \leq)$ and standard objects $\left\{ \, \Delta_r(\lambda)\, \middle| \,\lambda \in X_r^+ \, \right\}$ (see e.g. \cite[Theorem 3.8]{BLW13}). Since $X_r^{++}$ is an ideal in $X_r^+$ (an easy generalisation of \cite[Lemma 3.4]{CW07}), the proposition follows from the general theory of highest weight categories.
	\end{proof}
	
\begin{projcover}
	We will write $P_r(\lambda)$ for the projective cover of $L_r(\lambda)$ in $\mathcal{O}_r^{++}$. This will generally be a proper quotient of the projective cover of $L_r(\lambda)$ in the larger categories $\mathcal{O}_r^{+}$ and $\mc{O}_r$.
	\end{projcover}
	
We wish to extend this to $r=\infty$. Most of the necessary ingredients are already in the literature, it only remains to establish that $\mathcal{O}^{++}$ has enough projectives. Our main tool will be the truncation functors from Definition~\ref{def:trunc}. We will show that these send $P_s(\lambda)$ to $P_r(\lambda)$ for $\lambda\in X^{++}_r$ and use them to construct projective covers in $\mathcal{O}^{++}$ direct limits of the $P_r(\lambda)$ as in Lemma~\ref{limitinO}.

We will need a left adjoints $(\tr_r^s)^!$ to the $\tr_r^s$. Let $i_r^!:\catO_r\to\catO_r^{++}$ be the functor that sends a module to its largest quotient in $\catO_r^{++}$. It is left adjoint to the inclusion $i_r:\catO_r^{++}\to\catO_r$. Let $\mathfrak{p}_{r,1}$ be the parabolic subalgebra of $\mathfrak{g}_{r+1}$ corresponding to the Levi subalgebra $\mathfrak{g}_{r,1}:=\mathfrak{g}_r+ \mf{h}_{r+1}$. Take $M\in \cato_r^{++}$ and trivially extend the $\mf{g}_r$-action on M to an action of $\mf{p}_{r,1}$. Then $\mc{U}\mf{g}_{r+1}\otimes_{\mc{U}\mf{p}_{r,1}} M\in\mc{O}_{r+1}$. Define
\begin{equation}
	(\tr_r^{r+1})^!(M)=i_{r+1}^!\left(\mc{U}\mf{g}_{r+1}\otimes_{\mc{U}\mf{p}_{r,1}}M\right).
	\end{equation}

\noindent Now set
\begin{equation}
	(\tr_r^s)^!=(\tr_{s-1}^r)^!\circ\cdots (\tr_r^{r+1})^!:\cato_r^{++}\to\cato_s^{++}
	\end{equation}

\noindent for $r<s<\infty$.

\begin{tradj}
	If $r<s<\infty$ then $(\tr_r^s)^!$ is left adjoint to $\tr_r^s$.
	\end{tradj}

\begin{proof}
	It suffices to prove this in the case $s=r+1$. Take $M\in\cato_r^{++}$ and $N\in\cato_{r+1}^{++}$ and take a $\mathfrak{g}_{r,1}$-module homomorphism $f:M\to N$. We claim this is a homomorphism of $\mathfrak{p}_{r,1}$-modules. Indeed, since all composition factors of $N$ are of the form $L_{r+1}(\lambda)$ with $\lambda\in X_{r+1}^{++}$, all weights of $N$ must have non-negative $\delta_{m+n+r+1}$-component. As $f$ preserves weight spaces and the weight of any root vector in the nilradical of $\mathfrak{p}_{r,1}\subseteq\mathfrak{g}_{r+1}$ has $\delta_{m+n+r+1}$-component -1, the nilradical must act trivially on $\text{im}~f$and therefore $f$ is a homomorphism of $\mathfrak{p}_{r,1}$-modules. Now we have a chain of isomorphisms
	\begin{align}
		\begin{split}
		\Hom_{\mathfrak{g}_r}(M,\tr^{r+1}_r N)	& = \Hom_{\mathfrak{g}_{r,1}}(M,N) \\
										& = \Hom_{\mathfrak{p}_{r,1}}(M,N) \\
										& \cong \Hom_{\mathfrak{g}_{r+1}}(\mc{U}\mathfrak{g}_{r+1}\otimes_{\mc{U}\mathfrak{p}_{r,1}} M,N) \\
										& = \Hom_{\mathfrak{g}_{r+1}}((\tr_r^{r+1})^! M,N),
		\end{split}
		\end{align}
	
	\noindent where the penultimate isomorphism comes from the usual adjunction between induction and restriction.
	\end{proof}	

\begin{trproj} \label{trproj}
	If $r<s<\infty$ and $\lambda\in X_r^{++}$ then
	\begin{align}
		\tr_r^s \, P_s(\lambda)=P_r(\lambda).
		\end{align}
	\end{trproj}
	
\begin{proof}
	Without loss of generality assume $s=r+1$. Throughout the proof we write $\tr$ for $\tr^{r+1}_r$ and $\tr^!$ for $(\tr^{r+1}_r)^!$. First we claim $\tr^!\tr\tr^!=\tr^!$. Take $M\in\catO_r^{++}$. The weight of any root vector in the complement to $\mf{p}_{r,1}$ in $\mf{g}_{r+1}$ has $\delta_{m+n+r+1}$-component 1. So
	\begin{equation}
		\tr\left( \mc{U}\mf{g}_{r+1}\otimes_{\mc{U}\mf{p}_{r,1}}M\right)=M.
		\end{equation}
		
	\noindent Hence, since $\tr^!(M)$ is a quotient of $\mc{U}\mf{g}_{r+1}\otimes_{\mc{U}\mf{p}_{r,1}}M$, by exactness of $\tr$ there is a surjection $M\onto\tr\tr^!(M)$. Left adjoints are right exact so this induces a surjection ${\tr^!(M)\onto\tr^!\tr\tr^!(M)}$. But by the counit-unit equations, $\tr^!(M)$ is a direct summand of $\tr^!\tr\tr^!(M)$. Thus they are equal.
	
	Take $M=P_r(\lambda)$. A left adjoint to an exact functor sends projectives to projectives, so $\tr^! P_r(\lambda)$ is projective. If $\mu\in X_{r+1}^{++}$, the multiplicity of $P_{r+1}(\mu)$ as a direct summand of $\tr^! P_r(\lambda)$ equals
	\begin{equation}\label{eq:compmult}
		\dim \Hom_{\mathfrak{g}_{r+1}}( \tr^! P_r(\lambda),L_{r+1}(\mu)) = \dim \Hom_{\mathfrak{g}_r}\left( P_r(\lambda), \tr L_{r+1}(\mu)\right) = \delta_{\lambda \mu}
		\end{equation}
	
	\noindent by Proposition~\ref{tilt}. So $\tr^!P_r(\lambda)=P_{r+1}(\lambda)$. This implies ${\tr^!\tr P_{r+1}(\lambda)=P_{r+1}(\lambda)}$ and so there is an isomorphism of functors on $\catO_{r+1}^{++}$:
	\begin{equation}\label{eq:trff}
		\Hom_{\mf{g}_{r+1}}(P_{r+1}(\lambda),-)\cong \Hom_{\mf{g}_{r+1}}(\tr^!\tr P_{r+1}(\lambda),-)\cong \Hom_{\mf{g}_r}(\tr P_{r+1}(\lambda),\tr(-)).
		\end{equation}
	
	\noindent In particular the functor on the right is exact. On can show that this isomorphism is just the map induced by $\tr$.
	
	Now take $\mu\in X_r^{++}\subseteq X_{r+1}^{++}$. There is a short exact sequence
	\begin{equation}
		\begin{tikzcd}
			0 \ar[r]	&L_{r+1}(\mu) \ar[r]	& \nabla_{r+1}(\mu) \ar[r]	& C \ar[r]		& 0
			\end{tikzcd}
		\end{equation}
		
	\noindent and applying $\tr$ yields
	\begin{equation}
		\begin{tikzcd}
			0 \ar[r]	&L_r(\mu) \ar[r]	& \nabla_r(\mu) \ar[r]	& \tr(C) \ar[r]	& 0.
			\end{tikzcd}
		\end{equation}
	
	\noindent Consider the long exact sequence induced by $\Hom_{\mf{g}_r}(\tr P_{r+1}(\lambda),-)$. By \eqref{eq:trff} the $\Hom$ terms form a short exact sequence and so there is an injection
	\begin{equation}
		\Ext^1_{g_r}(\tr P_{r+1}(\lambda),L_r(\mu))\into\Ext^1_{g_r}(\tr P_{r+1}(\lambda),\nabla_r(\mu)).
		\end{equation}
		
	\noindent But $P_{r+1}(\lambda)$ has a $\Delta$-flag, so $\tr P_{r+1}(\lambda)$ does also, and therefore ${\Ext^1_{g_r}(\tr P_{r+1}(\lambda),\nabla_r(\mu))=0}$. So ${\Ext^1_{g_r}(\tr P_{r+1}(\lambda),L_r(\mu))=0}$ and now induction on length shows ${\Ext^1_{g_r}(\tr P_{r+1}(\lambda),N)=}0$ for any $N\in\catO_r^{++}$. So $\tr P_{r+1}(\lambda)$ is projective. Arguing as in \eqref{eq:compmult} shows ${\tr P_{r+1}(\lambda)=P_r(\lambda)}$.
	\end{proof}
	
\begin{pinftydef}\label{pinftydef}
	Take $\lambda\in X^{++}$. If $s>r\gg 0$ then $\lambda\in X_r^{++}$ so by Proposition \ref{trproj} there is an inclusion of $\mathfrak{g}_r$-modules $P_r(\lambda)=\tr^s_rP_s(\lambda)\into P_s(\lambda)$. Define a $\mf{g}=\mf{g}_{\infty}$-module $P(\lambda)$ by
	\begin{equation}
		P(\lambda)=\lim_{\longrightarrow} \, P_r(\lambda).
		\end{equation}
	\end{pinftydef}
 
 \begin{pinftythm} \label{pinftythm}
 	Take $\lambda\in X^{++}$ and let $r_{\lambda}\in\NN$ be minimal such that $\lambda\in X_{r_\lambda}^{++}$ and $1-r_\lambda\leq (\lambda+\rho_{r_\lambda},\delta_i)\leq r_\lambda$ for all $i$. Then
	\begin{enumerate}[label=(\roman*)]
		\item \label{deltaflag} the $\Delta$-multiplicities of $P_r(\lambda)$ are independent of $r\geq r_{\lambda}$ in the sense of \eqref{eq:multstable}, so $P(\lambda)\in\cato^{++}$ by Remark~\ref{rmk:deltamultstable};
		\item \label{projcover} $P(\lambda)$ is a projective cover of $L(\lambda)$ in $\cato^{++}$.
		\end{enumerate}
	
	\noindent In particular, $\cato^{++}$ has enough projectives.
	\end{pinftythm}

\begin{Rmk:Chenproj}
	While drafting this paper the author was made aware the Chih-Whi Chen and Ngau Lam in Taiwan have independently proved that $\catO^{++}$ has enough projectives using a different approach. 
	\end{Rmk:Chenproj}

\begin{proof}
	\ref{deltaflag} Take $s>r\geq r_{\lambda}$. Applying $\tr^s_r$ to a $\Delta$-flag of $P_s(\lambda)$ yields a $\Delta$-flag of $P_r(\lambda)$. In light of Proposition~\ref{tilt} it suffices to show that if $\mu\in X_s^{++}$ and $\left( P_s(\lambda):\Delta_s(\mu)\right)\neq 0$ then $\mu\in X_r^{++}$. But if $\left( P_s(\lambda):\Delta_s(\mu)\right)\neq 0$ then $\mu\geq \lambda$. Thus it suffices to show that if $\mu\in X^{++}$ with $\mu\geq\lambda$ then $\mu\in X_{r_{\lambda}}^{++}$.

	Assume $\parity=0$ (the case $\parity=1$ is similar). Take $r$ minimal such that $\mu\in X_r^{++}$. Suppose for contradiction that $r>r_\lambda$. We have $r,r_{\lambda}\in X_r^{++}$ so by \cite[Lemma 3.9]{BLW13},
	
	\begin{equation}\label{eq:localcoideal}
		\sum_{\substack{1\leq i\leq m+n+r\\ (\lambda+\rho_r,\delta_i)\leq 1-r}} (-1)^{p_i} = \sum_{\substack{1\leq i\leq m+n+r\\(\mu+\rho_r,\delta_i)\leq 1-r}} (-1)^{p_i}.
		\end{equation}
	
	\noindent Since $r>r_{\lambda}$, $(\lambda+\rho_r,\delta_i)>1-r$ unless $i=m+n+r$ so the left hand side is equal to $(-1)^{\epsilon}=1$. Similarly, $\mu\in X_r^{++}\setminus X_{r-1}^{++}$ implies $(\mu+\rho_r,\delta_i)>1-r$ for $i>m+n$. So the right hand side of \eqref{eq:localcoideal} doesn't change if we restrict the sum to be over $1\leq i\leq m+n$. But
	\begin{equation}
			\sum_{\substack{1\leq i\leq m+n\\(\mu+\rho_r,\delta_i)\leq 1-r}} (-1)^{p_i} \leq \sum_{\substack{1\leq i\leq m+n\\(\lambda+\rho_r,\delta_i)\leq 1-r}} (-1)^{p_i},
		\end{equation}
		
	\noindent again by \cite[Lemma 3.9]{BLW13} and the right hand side is equal to 0, a contradiction. So $\mu\in X_{r_{\lambda}}^{++}$ and \ref{deltaflag} holds.
	\newline
		
	\noindent \ref{projcover} We claim that $P(\lambda)$ is projective. Take a diagram
	
	\begin{equation} \label{projdiag}
		\begin{tikzcd}
						& P(\lambda) \ar[d, "f"] 	& \\
			M \arrow[r, "g"]	& N \ar[r]				& 0	
			\end{tikzcd}
		\end{equation}
				
	\noindent Without loss of generality assume $f\neq 0$. Let $\mu_1,\ldots ,\mu_k\in X_{r_\lambda}^{++}$ be the ordered sequence of weights in a $\Delta$-flag of $P(\lambda)$ and take weight vectors $v_1,\ldots ,v_k\in P(\lambda)$ such that $v_i$ projects onto the highest weight vector of weight $\mu_i$ in the subquotient $\Delta(\mu_i)$ of $P(\lambda)$. Then $v_1,\ldots, v_k$ lie in $P_r(\lambda)$ for any $r\geq r_{\lambda}$ by part \ref{deltaflag}.
			
	If $r\geq r_\lambda$, then applying $\tr_r^\infty$ to \eqref{projdiag} and using projectivity of $P_r(\lambda)$ yields a commutative diagram:
	\begin{equation}
		\begin{tikzcd}
							& P_r(\lambda) \ar[ld, "h_r"'] \ar[d, "f_r"] 	& \\
			M_r \ar[r, "g_r"']		& N_r \ar[r]						& 0	
			\end{tikzcd}
		\end{equation}			
			
	\noindent where $M_r=\tr^{\infty}_r(M)$, $N_r=\tr^{\infty}_r(N)$, and $f_r$ and $g_r$ denote the restrictions of $f$ and $g$ respectively. Write $\underline{v}=(v_1,\ldots,v_k)$ and let
	\begin{equation}
		A_r=\text{span} \big\{ \, h_s(\underline{v})=(h_s(v_1),\ldots ,h_s(v_k)) \, \big| \, s\geq r \, \big\}\leq M_{\mu_1}\times \cdots \times M_{\mu_n},
		\end{equation}
	
	\noindent a finite-dimensional vector space with
	\begin{equation}
		\cdots \subseteq A_{r+1}\subseteq A_r\subseteq \cdots
		\end{equation}
			
	\noindent If $\underline{w}=(w_1,\ldots,w_k)={\sum_{s\geq r} a_s h_s(\underline{v})} \in A_r$ then $\tilde{h}_r={\sum_{s\geq r} a_s h_s\big|_{P_r(\lambda)}}:P_r(\lambda)\to M_r$ is a $\mf{g}_r$-module map with $\tilde{h}_r(v_i)=w_i$. It is unique with this property. Moreover,
	\begin{equation} \label{gh=f}
		g(\tilde{h}_r(v_i))=\sum_{s\geq r} a_sg(h_s(v_i))=\Big(\sum_{s\geq r} a_s\Big) f(v_i)
		\end{equation}
			
	\noindent for all $i$ so that in particular
	\begin{equation}
		\sum_{s\geq r} a_sh_s(v_i)=\sum_{s\geq r} b_sh_s(v_i)\in A_r \quad \Rightarrow \quad \sum_{s\geq r}a_s=\sum_{s\geq r}b_s,
		\end{equation}
			
	\noindent since $f\neq 0$ implies that there exists an $i$ with $f(v_i)\neq 0$. We wish to find $\underline{w}\in\bigcap_{r\geq r_{\lambda}} A_r$ with $\sum_{s\geq r}a_s=1$ and define $h:P(\lambda)\to M$ by $h\big|_{P_r(\lambda)}=\tilde{h}_r$.
			
	For $r\geq r_\lambda$ let
	\begin{equation}
		B_r= \Big\{ \, \sum_{s\geq r} a_s h_s(\underline{v}) \, \Big| \, \sum_{s\geq r}a_s=0 \, \Big\} \leq A_r.
		\end{equation}
			
	\noindent Since all spaces are finite-dimensional, the short exact sequence of inverse systems
	\begin{equation}
		\begin{tikzcd}
			0 \ar[r]		& B_{r+1} \ar[d] \ar[r]		& A_{r+1} \ar[d] \ar[r]		& \left. A_{r+1}\middle/ B_{r+1}\right. \ar[d] \ar[r]		& 0 \\
			0 \ar[r]		& B_r \ar[r]			& A_r \ar[r]			& \left. A_r\middle/ B_r\right. \ar[r]				& 0
			\end{tikzcd}
		\end{equation}
			
	\noindent induces a short exact sequence of the inverse limits:
	\begin{equation}
		\begin{tikzcd}
			0 \ar[r]	& \displaystyle{\bigcap_{r\geq r_\lambda} B_r} \ar[r]	& \displaystyle{\bigcap_{r\geq r_\lambda} A_r} \ar[r]	& \displaystyle{\lim_{\longleftarrow}} \, \left. A_r \middle/ B_r\right. \ar[r]		& 0.
			\end{tikzcd}
		\end{equation}
			
	\noindent Each $A_r \big/ B_r$ is a non-zero, finite-dimensional vector space, so $\displaystyle{\lim_{\longleftarrow}} \, A_r \big/ B_r \neq 0$ and so $\bigcap B_r \subsetneq \bigcap A_r$. Take $\underline{w}\in \bigcap  A_r \setminus \bigcap B_r$ such that ${\underline{w}=\sum a_s h_s(\underline{v})}$ implies $\sum a_s=1$. The assignment $v_i\mapsto w_i$ induces a well-defined $\mf{g}_r$-module homomorphism $P_r(\lambda)\to M_r$ for any $r\geq r_\lambda$ and so induces a well-defined $\mf{g}$-module homomorphism $h:P(\lambda)\to M$. Morever, $g(h(v_i))=f(v_i)$ for all $i$ by \eqref{gh=f} and so $g\circ h=f$. Thus $P(\lambda)$ is projective.
			
	Now we claim that $P(\lambda)$ is a projective cover of $L(\lambda)$. Indeed, from the $\Delta$-flag of $P(\lambda)$ there is a surjection $P(\lambda)\onto \Delta(\lambda)$ and so there is an epimorphism $\pi:P(\lambda)\onto L(\lambda)$. We claim that $\pi$ is superfluous. Suppose $M\in \cato^{++}$ and $f:M\to P(\lambda)$ with $\pi\circ f$ surjective. Then $\tr^{\infty}_r(\pi\circ f)=\tr^{\infty}_r(\pi)\circ f_r$ is surjective for any $r\geq r_{\lambda}$ and so $f_r$ is surjective since $\tr^{\infty}_r(\pi):P_r(\lambda)\onto L_r(\lambda)$ is a projective cover. But $f=\bigcup f_r$, so $f$ is surjective and thus $\pi$ is superfluous.
	\end{proof}
	
\begin{O++schurcat}
	$\cato^{++}$ is a highest weight category with weight poset $\left( X^{++} \!, \leq \right)$ and standard objects $\left\{ \, \Delta(\lambda) \, \middle| \, \lambda \in X^{++} \, \right\}$.
	\end{O++schurcat}
	
\begin{proof}
	The category $\mathcal{O}^{++}$ is abelian by \cite[\S7.2]{CLW12}, objects have finite length by definition, and $\dim \End(L(\lambda))=1$ for any $\lambda\in X^{++}$. By the theorem, $\cato^{++}$ has enough projectives and applying duality shows it has enough injectives. Thus $\cato^{++}$ is a Schurian category. The poset $\left( X^{++} \!, \leq \right)$ is interval-finite by Lemma 2.4 in \textit{loc. cit.} and the remaining highest weight conditions follow easily from the theorem.
	\end{proof}
	
\subsection{Categorical action}\label{subsec:cataction}

\begin{not:Mr}
	If $M\in\catO^{++}$ and $r\in\NN$ then we will write $M_r:=\tr^{\infty}_r(M)$.
	\end{not:Mr}

\noindent For $r<\infty$ let
\begin{equation}
	\Omega_r:=\sum_{k,l=1}^{m+n+r} (-1)^{p_l}e_{kl}\otimes e_{lk}.
	\end{equation}

\noindent For $M_r\in\mathcal{O}_r$, let $F_rM_r =M_r\otimes U_r$ and $E_rM_r=M_r\otimes U_r^*$. Let $x_r\in \End\left(F_r\right)$ be given by multiplication by $\Omega_r$ and $t_r\in\End\left(F_r^2\right)$ be induced by the map
\begin{equation}
	\begin{aligned}
		U_r\otimes U_r	& \longrightarrow	 U_r\otimes U_r \\
		u_k\otimes u_l	& \longmapsto		 (-1)^{p_kp_l} u_l\otimes u_k.
		\end{aligned}
	\end{equation}
	
\begin{finranktpc}\label{finranktpc}
	\cite[Theorem 3.10]{BLW13} With respect to the above actions, $\mathcal{O}_r$ is an $\mf{sl}_{I_r}$-TPC of type $(\ur,\ud)$.
	\end{finranktpc}

\noindent The fact that the data above defines strong $\mf{sl}_2$-categorifications on $\catO_r$  \`{a} la Chuang-Rouquier \cite{CR04} was checked in \cite[Proposition~5.1]{CW07}.
	
The endomorphism $x_r\in\End\left(F_r\right)$ induces an endomorphism of $E_r$, also denoted $x_r$, given by multiplication by $\left(m-n-(-1)^{\parity}r\right)-\Omega_r$. For $j\in\mathbb{Z}$, let $F_{j,r}$ and $E_{j,r}$ denote the generalized $j$-eigenspaces of $x_r$ on $F_r$ and $E_r$ respectively.

For $\lambda\in X_r$, $j\in\ZZ$, and $1\leq i\leq l+1$, let $t^i_{j}(\lambda)\in X_r$ be obtained from $\lambda$ by applying the transposition $(j~j+1)$ to the $i^{\text{th}}$ row of $\lambda$, considered as a 01-matrix $(\lambda^i_j)$ as described in \S\ref{subsec:Set-up}. Then $F_{j,r}\Delta_r(\lambda)$ has a $\Delta$-flag and
\begin{equation}\label{FDeltamult}
	\left[ F_{j,r}\Delta_r(\lambda)\right] =\sum_i \left[ \Delta_r(t^i_j(\lambda))\right]
	\end{equation}
	
\noindent where the sum is taken over all $1\leq i\leq l+1$ with $\lambda^i_{j}=1$ and $\lambda^i_{j+1}=0$. Similarly $E_{j,r}\Delta_r(\lambda)$ has a $\Delta$-flag and
\begin{equation}\label{EDeltamult}
	\left[ E_{j,r}\Delta_r(\lambda)\right] =\sum_i \left[ \Delta_r(t^i_j(\lambda))\right].
	\end{equation}

\noindent where the sum is taken over all $1\leq i\leq l+1$ with $\lambda^i_{j}=0$ and $\lambda^i_{j+1}=1$. 

If $\lambda\in X_r^{++}$ and $r>\left| j\right|$ then $t^i_j(\lambda)\in X_r^{++}$ so $E_{j,r}$ and $F_{j,r}$  restrict to endofunctors of $\mathcal{O}_r^{++}$.

\begin{stablemult} \label{stablemult}
	Take $M\in \mathcal{O}^{++}$ and $j\in\ZZ$. There exists $r_M>\lvert j\rvert$ such that the composition multiplicities of $E_{j,r}M_r$ and $F_{j,r}M_r$ are independent of $r\geq r_M$ in the sense of \eqref{eq:compmult}. We will always assume that $r_M>\lvert j-m+n\rvert$.
	\end{stablemult}
	
\noindent Of course $r_M$ depends on $j$ as well as $M$. But since we rarely vary $j$ we won't record this dependence in our notation.
	
\begin{proof}
	Take $\lambda\in X^{++}$. It suffices to prove the claim for $M=L(\lambda)$. Take $r_M>\lvert j\rvert$ such that $\lambda\in X_{r_M}^{++}$ and if $\mu\in X^{++}$ with $[ \Delta (t^i_j(\lambda)):L(\mu)]\neq 0$ for some $1\leq i\leq l+1$ then $r_M\geq r_{\mu}$ (c.f. Theorem~\ref{pinftythm}).
	
	Take $r\in\NN$ and suppose $\mu\in X_r^{++}$ with $[ F_{r,j}L_r(\lambda):L_r(\mu)]\neq 0$. Since $\Delta_r(\lambda)\onto L_r(\lambda)$, ${[ F_{r,j}\Delta_r(\lambda):L_r(\mu)]\neq 0}$ and so $[ \Delta_r(t^i_j(\lambda)):L_r(\mu)]\neq 0$ for some $1\leq i\leq l+1$ by \eqref{FDeltamult}. This implies $[ \Delta(t^i_j(\lambda)):L(\mu)]\neq 0$ and so $\mu\in X_{r_{\mu}}^{++}\subseteq X_{r_M}^{++}$ by the definition of $r_M$. We have
	\begin{equation*}
		[ F_{r,j}L_r(\lambda):L_r(\mu)] = \dim\Hom_{\mf{g}_r}(P_r(\mu),F_{j,r}L_r(\lambda)) = \dim\Hom_{\mf{g}_r}(E_{j,r}P_r(\mu),L_r(\lambda)),
		\end{equation*}
		
	\noindent which is the multiplicity of $P_r(\lambda)$ as a direct summand of $E_{j,r}P_r(\mu)$. As $r_M\geq r_{\mu}$, Theorem~\ref{pinftythm}\ref{deltaflag} implies that the $\Delta$-multiplicities of $P_r(\mu)$ are independent of $r\geq r_M$. So the same is true of $E_{j,r}P_r(\mu)$ by \eqref{EDeltamult}. The decomposition of $E_{j,r}P_r(\mu)$ into indecomposables is uniquely determined by these multiplicities and so $[ F_{r,j}L_s(\lambda):L_r(\mu)]$ is independent of $r\geq r_M$. The analogous proof works for $E_{j,r}M$.
	\end{proof}

For $r=\infty$ we define $\Omega$, $F$, $F_j$, $x$, and $t$ analogously. If $M\in \catO^{++}$ and $r\in\mathbb{N}$ then $\left(FM\right)_r=F_rM_r$. Moreover, the action of $\Omega$ on $FM$ restricts to the action of $\Omega_r$ on $F_rM_r$ so if $r>\left|j\right|$ then $(F_jM)_r=F_{j,r}M_r$. By the proposition above and Lemma~\ref{limitinO}, $F_jM\in\catO^{++}$. By \eqref{FDeltamult}, only finitely many $F_jM$ are non-zero, and so $FM=\bigoplus_j F_jM\in\catO^{++}$. So $F$ and $F_j$ are well-defined endofunctors of $\catO^{++}$. 

It remains to define a two-sided adjoint $E$ to $F$. In general, $M\otimes U^*\notin \mc{O}^{++}$ for $M\in\mc{O}^{++}$ so the obvious definition won't work. This is because Proposition~\ref{stablemult} doesn't hold if we replace $E_{j,r}M_r$ with $E_rM_r$. Instead we will define each $E_jM$ as a direct limit of the $E_{j,r}M_r$ and then set $EM=\bigoplus_j E_jM$. However, the natural inclusion $E_rM_r\into E_{r+1}M_{r+1}$ doesn't restrict to inclusions $E_{j,r}M_r\into E_{j,r+1}M_{r+1}$. To get around this we will define two sets of maps $E_{j,r}M_r\to \tr^{r+1}_r(E_{j,r+1}M_{r+1})$, leading to functors $E_j^{\texttt{L}}$ and $E_j^{\texttt{R}}$ on $\catO^{++}$ that are naturally left and right adjoint $F_j$ respectively. Finally we will show that $E_j^{\texttt{L}}\cong E_j^{\texttt{R}}$.

\begin{Ejleft}\label{Ejleft}
	Take $j\in \ZZ$ and $M\in\mathcal{O}^{++}$. For $r\in\mathbb{N}$, let $\psi_r$ be the composition of $\mf{g}_r$-module homomorphisms below:
	\begin{equation}
		E_{j,r}M_r\subseteq E_rM_r\subseteq E_{r+1}M_{r+1}\onto E_{j,r+1}M_{r+1}.
		\end{equation}
	
	\noindent Define $E_j^{\texttt{L}}M=\displaystyle{\lim_{\longrightarrow}} \,E_{j,r}M_r$, where the limit is taken over the maps $\psi_r$ above.
	\end{Ejleft}

\begin{EjLinO}\label{EjLinO}
	Take $M\in\catO^{++}$ and $r\geq r_M$ (see Proposition~\ref{stablemult}). Then $\psi_r$ is an injective $\mf{g}_r$-module homomorphism $E_{j,r}M_r\to\tr^{r+1}_r(E_{j,r+1}M_{r+1})$ and so $E_j^{\texttt{L}}M\in\catO^{++}$ by Proposition~\ref{stablemult} and Lemma~\ref{limitinO}.
	\end{EjLinO}

\begin{proof}	
	Take $\alpha\in E_rM_r$ and $d\in\ZZ$. For $t\in\NN$, define
	\begin{equation}
		\beta_t:=(d-(-1)^{\parity}-\Omega_{r+1})^t\alpha-(d-\Omega_r)^t\alpha.
		\end{equation}
			
	\noindent We claim $\Omega_{r+1}\beta_t=0$. Proceed by induction on $t$. First take $t=1$. We have
	\begin{equation}
		\beta_1=\Omega_r\alpha-\Omega_{r+1}\alpha -(-1)^{\parity}\alpha
		\end{equation}
	
	\noindent Let $z:=m+n+r+1$. Write $\alpha=\sum_{i=1}^{z-1}m_i\otimes u_i^*$ with $m_i\in M_r$. By direct computation,
	\begin{align}\label{eq:betainduction1}
		\begin{split}
			\Omega_{r+1}\alpha	& =-\sum_{k=1}^z\left(\sum_{i=1}^{z-1} (-1)^{p_i} e_{ki}m_i\right)\otimes u_k^* \\
							& = \Omega_r\alpha-\left(\sum_{i=1}^{z-1} (-1)^{p_i} e_{zi}m_i\right)\otimes u_z^*,
			\end{split}
		\end{align}
		
	\noindent so
	\begin{equation}\label{eq:betainduction2}
		\Omega_{r+1}(\Omega_r\alpha-\Omega_{r+1}\alpha)=-\sum_{k=1}^z\sum_{i=1}^{z-1}(-1)^{p_i+\parity}e_{kz}e_{zi}m_i\otimes u_k^*.
		\end{equation}
	
	\noindent By the definition of the supercommutator bracket,
	\begin{align}\label{eq:supercomm}
		\begin{split}
			e_{kz}e_{zi}m_i	& = (-1)^{(p_i+\parity)(p_k+\parity)}e_{zi}e_{kz}m_i+\left[ e_{kz},e_{zi}\right]m_i \\
						& = (-1)^{(p_i+\parity)(p_k+\parity)}e_{zi}e_{kz}m_i+e_{ki}m_i-(-1)^{p_i+\parity}\delta_{ik}e_{zz}m_i.
			\end{split}
		\end{align}
	
	\noindent If $k<z$ then applying $e_{kz}$ to a weight vector in $M_r$ yields an element of $M_{r+1}$ whose weight has $\delta_z$-component -1. Since $M_{r+1}\in\catO_{r+1}^{++}$, this weight space is zero. So $e_{kz}m_i=0$. By weight considerations, $e_{zz}m_i=0$ also. Therefore \eqref{eq:betainduction2} equals
	\begin{equation}
		-\sum_{k=1}^z\sum_{i=1}^{z-1}(-1)^{p_i+\parity} e_{ki}m_i\otimes u_k^*=(-1)^{\parity} \Omega_{r+1}\alpha.
		\end{equation}
		
	\noindent The claim follows.
	
	Now suppose $\Omega_{r+1}\beta_t=0$ for some $t\in\NN$. Let
	\begin{equation}
		\beta_t':=\left( d-(-1)^{\parity}-\Omega_{r+1}\right)^t(d-\Omega_r)\alpha-\left(d-\Omega_r\right)^{t+1}\alpha.
		\end{equation}
	
	\noindent Note that $\Omega_{r+1}\beta_t'=0$ by the inductive hypothesis applied to $(d-\Omega_r)\alpha\in E_rM_r$. But
	\begin{align}
		\begin{split}
			(d-(-1)^{\parity}-\Omega_{r+1})^{t+1}\alpha	& = \left(d-(-1)^{\parity}-\Omega_{r+1}\right)^t\left(d-(-1)^{\parity}-\Omega_{r+1}\right)\alpha \\
												& = \left(d-(-1)^{\parity}-\Omega_{r+1}\right)^t\left((d-\Omega_{r})\alpha+\beta_1\right) \\
												& = (d-\Omega_r)^{t+1}\alpha+\left( \beta_t'+(d-(-1)^{\parity})^t\beta_1\right).
			\end{split}
		\end{align}
	
	\noindent So $\beta_{t+1}=\beta_t'+(d-(-1)^{\parity})^t\beta_1\in\Ker~\Omega_{r+1}$.
	
	Now take $\alpha\in E_{j,r}M_r$. The module $E_{j,r}M_r$ is defined to be the generalized $j$-eigenspace of ${(m-n-(-1)^{\parity}r)-\Omega_r}$, so there exists $t\in\NN$ such that
	\begin{equation}\label{eq:alphainespace}
		(d-\Omega_r)^t\alpha=0,
		\end{equation}
	
	\noindent where $d=m-n-(-1)^{\parity}r-j$. So
	\begin{equation}
		\beta_t=(d-(-1)^{\parity}-\Omega_{r+1})^t\alpha\in\Ker~\Omega_{r+1}.
		\end{equation}
	
	\noindent Define $\beta$ by the equation $(d-(-1)^{\parity}-\Omega_{r+1})^t\beta=\beta_t$. Then $\Omega_{r+1}\beta=0$ and so
	\begin{equation}
		\left( d-(-1)^{\parity}-\Omega_{r+1}\right)^t(\alpha-\beta)=\beta_t-(d-(-1)^{\parity})^t\beta=0.
		\end{equation}
		
	\noindent So $\alpha =(\alpha-\beta)+\beta$ is a decomposition of $\alpha\in E_rM_r\subseteq E_{r+1}M_{r+1}$ into generalized $(d-(-1)^{\parity}-\Omega_{r+1})$-eigenvectors and $\alpha-\beta\in E_{j,r+1}M_{r+1}$. So $\psi_r(\alpha)=\alpha-\beta$. 
	
	Suppose $\psi_r(\alpha)=0$. Then $\alpha=\beta\in\Ker~\Omega_{r+1}$ so \eqref{eq:betainduction1} implies ${\sum_{i=1}^{z-1} (-1)^{p_i} e_{ki}m_i=0}$ for each $k$ and thus $\Omega_r\alpha=0$. Generalized eigenspaces with different eigenvalues intersect trivially, so if $\alpha\neq 0$ then \eqref{eq:alphainespace} implies $d=0$, so $j=m-n-(-1)^{\parity}r$. But this contradicts the assumption $r\geq r_M>\lvert j-m+n\rvert$ made in Proposition~\ref{stablemult}. So $\psi_r$ is injective.
	\end{proof}

Now we define the right adjoint $E_j^{\texttt{R}}$ to $F_j$. Take $M\in\catO^{++}$, $r\geq r_M$, and an element ${x\in\tr^{r+1}_r(E_{j,r+1}M_{r+1})}$. We can write
\begin{equation}
	x=\sum_{i=1}^{m+n+r+1} m_i\otimes u_i^*
	\end{equation}
	
\noindent where $m_i\in M_r$ for $1\leq i\leq m+n+r+1$ and $m_{m+n+r+1}\in M_{r+1}$ is a sum of weight vectors whose weights have $\delta_{m+n+r+1}$-component 1. Define
\begin{equation}\label{phidef}
	\begin{aligned}
		\phi_r:\tr^{r+1}_r(E_{j,r+1}M_{r+1}) & 	\longrightarrow E_rM_r \\
		\sum_{i=1}^{m+n+r+1} m_i\otimes u_i^* &		\longmapsto \sum_{i=1}^{m+n+r} m_i\otimes u_i^*
		\end{aligned}
	\end{equation}

\noindent This is a homomorphism of $\mathfrak{g}_r$-modules.

\begin{phi11}\label{phi11}
	Take $M\in\catO$ and $r\geq r_M$. Then 
	\begin{enumerate}[label=(\roman*)]
		\item \label{phi11image}$\text{im}~\phi_r\subseteq E_{j,r}M_r$;
		\item \label{phi11inj}$\phi_r$ is injective and so is an isomorphism by Proposition \ref{stablemult}.
		\end{enumerate}
	\end{phi11}

\begin{Ejright}
	Let $E_j^{\texttt{R}}M=\displaystyle{\lim_{\longrightarrow}}~E_{j,r}M_r$, where the direct limit is taken along the maps
	\begin{equation}
		E_{j,r}M_r \stackrel{\phi_r^{-1}}{\longrightarrow} \tr^{r+1}_r(E_{j,r+1}M_{r+1})\subseteq E_{j,r+1}M_{r+1}
		\end{equation}
	
	\noindent for $r\geq r_M$. We have $E_j^{\texttt{R}}M\in\catO^{++}$ by Proposition~\ref{stablemult} and Lemma~\ref{limitinO}.
	\end{Ejright}

\begin{proof}[Proof of Lemma~\ref{phi11}]
	Take $x\in\tr^{r+1}_r(E_{j,r+1}M_{r+1})$ and write $x=\sum_{i=1}^z m_i\otimes u_i^*$, where $z=m+n+r+1$. Fix $d\in\ZZ$. For $t\geq 0$, write
	\begin{equation}\label{eq:mkt}
		\left(d-(-1)^{\parity}-\Omega_{r+1}\right)^tx=\sum_{k=1}^z m_{kt}\otimes u_k^*
		\end{equation}
	
	\noindent so that $m_{k0}=m_k$ for $1\leq k\leq z$. The map $\Omega_{r+1}$ preserves $E_{j,r+1}M_{r+1}$ and preserves weights, so \eqref{eq:mkt} lies in $\tr^{r+1}_r(E_{j,r+1}M_{r+1})$.
	
	We claim that 
	\begin{equation}\label{eq:mktclaim}
		m_{k,t+1}-e_{kz}m_{z,t+1}=(d-(-1)^{\parity})\left(m_{kt}-e_{kz}m_{zt}\right)
		\end{equation}
		
	\noindent for all $1\leq k\leq z$ and $t\geq 0$. By direct computation,
	\begin{equation}\label{eq:mktrecursion}
		m_{k,t+1}=(d-(-1)^{\parity})m_{kt}+\sum_{i=1}^z(-1)^{p_i}e_{ki}m_{it}.
		\end{equation}
	
	\noindent So
	\begin{align}
		\begin{split}
			m_{k,t+1}-e_{kz}m_{z,t+1}	=	&   \,(d-(-1)^{\parity})m_{kt}+\sum_{i=1}^z(-1)^{p_i}e_{ki}m_{it} \\
									& - (d-(-1)^{\parity})e_{kz}m_{zt}+\sum_{i=1}^z(-1)^{p_i}e_{kz}e_{zi}m_{it}.
			\end{split}
		\end{align}
	
	\noindent But $e_{kz}e_{zz}m_{zt}=e_{kz}m_{zt}$ by the weight of $m_{zt}$, and if $1\leq i\leq z-1$ then $e_{kz}e_{zi}m_{it}=e_{ki}m_{it}$ as in \eqref{eq:supercomm}. The claim follows.
	
	Set $d=m-n-(-1)^{\parity}r-j$. Since $x\in E_{j,r+1}M_{r+1}$, we have ${\left( d-(-1)^{\parity}-\Omega_{r+1}\right)^tx=0}$ for $t\gg 0$ and so $m_{kt}=0$ for $1\leq k\leq z$ and $t\gg0$. As $r\geq r_M$, the assumption $r_M>\lvert j-m+n\rvert$ in Proposition~\ref{stablemult} means $d-(-1)^{\parity}\neq0$ and so \eqref{eq:mktclaim} implies that $e_{kz}m_{zt}=m_{kt}$ for all $k$ and $t$. Now \eqref{eq:mktrecursion} simplifies to
	\begin{equation}\label{eq:newrecursion}
		m_{k,t+1}=dm_{k,t}+\sum_{i=1}^{z-1}e_{ki}m_{i,t}.
		\end{equation}
	
	\noindent But for $1\leq k\leq z-1$, this is the same recursive formula as for the terms of $(d-\Omega_r)^t\phi_r(x)$. So
	\begin{equation}
		\phi_r\left( (d-(-1)^{\parity}-\Omega_{r+1})^rx\right) = (d-\Omega)^t\phi_r(x)
		\end{equation}
	
	\noindent In particular this implies that $(d-\Omega)^t\phi_r(x)=0$ for $t\gg 0$, so $\phi_r(x)\in E_{j,r}M_r$ and \ref{phi11image} holds.
	
	Now suppose $\phi_r(x)=0$. Then $m_i=0$ for $1\leq i\leq z-1$. Equation \eqref{eq:newrecursion} shows that $m_{kt}=0$ for $1\leq k\leq z-1$ and $m_{zt}=d^tm_z$ for all $t\geq 0$. But $m_{zt}=0$ for $t\gg 0$ and $d\neq0$ so this implies $m_z=m_{z0}=0$, and therefore $x=0$.
	\end{proof}

\begin{Fjadj}
	Take $j\in\mathbb{Z}$. Then
	
	\begin{enumerate}[label=(\roman*)]
		\item\label{Ladj} $E_j^{\texttt{L}}$ is left-adjoint to $F_j$;
		\item\label{Radj} $E_j^{\texttt{R}}$ is right-adjoint to $F_j$;
		\item\label{L=R} $E_j^{\texttt{L}}\cong E_j^{\texttt{R}}$
		\end{enumerate}
	\end{Fjadj}

\begin{proof}
	\ref{Ladj} If $M,N\in\catO^{++}$ then
	\begin{align}
		\begin{split}
			\Hom_{\mf{g}}( M,F_jN)	& = \Hom_{\mf{g}}( \lim_{\longrightarrow}M_r,F_jN) \\
								& = \lim_{\longleftarrow} \Hom_{\mf{g}_r}( M_r,F_{j,r}N_r) \\
								& = \lim_{\longleftarrow} \Hom_{\mf{g}_r}( E_{j,r}M_r,N_r) \\
								& =\Hom_{\mf{g}}(\lim_{\longrightarrow} E_{j,r}M_r,N) \\
			\end{split}
		\end{align}
	
	\noindent Unravelling these isomorphisms shows that the connecting maps in the final direct limit $\displaystyle{\lim_{\longrightarrow}}~E_{j,r}M_r$ are just the $\psi_r$ from Definition~\ref{Ejleft}. The result follows.
	\newline
	
	\noindent\ref{Radj} Similar.
	\newline
	
	\noindent\ref{L=R} Take $M\in\catO^{++}$ and $r\geq r_M$. Multiplication by $\Omega_r$ is a $\mf{g_r}$-module isomorphism when restricted to $E_{j,r}M_r$, so it suffices to show that the following diagram commutes:
	\begin{equation}
		\begin{tikzcd}
			E_{j,r}M_r \ar[d, "\psi_r"] \ar[r, "\Omega_r"]		& E_{j,r}M_r \ar[d, leftarrow, "\phi_r"] \\
			\tr^{r+1}_r(E_{j,r+1}M_{r+1}) \ar[r, "\Omega_{r+1}"]	& \tr^{r+1}_r(E_{j,r+1}M_{r+1})
			\end{tikzcd}
		\end{equation}
		
	Take $\alpha\in E_{j,r}M_r$. From the proof of Lemma~\ref{EjLinO}, $\psi_r(\alpha)=\alpha-\beta$ for some ${\beta\in \Ker~\Omega_{r+1}}$. So $\Omega_{r+1}\psi_r(\alpha)=\Omega_{r+1}\alpha$. Equation \eqref{eq:betainduction1} implies $\phi_r(\Omega_{r+1}\alpha)=\Omega_r\alpha$, so $\phi_r(\Omega_{r+1}\psi_r(\alpha))=\Omega_r\alpha$ as required.
	\end{proof}

Write $E_j=E_j^{\texttt{R}}$ and let $E=\bigoplus_j E_j$. If $M\in\catO^{++}$ then only finitely many $E_jM$ are non-zero by $\eqref{EDeltamult}$, so $EM\in\catO^{++}$. The biadjunctions between the $F_j$ and $E_j$ induce a biadjunction between $F$ and $E$.

\begin{infranktpc}\label{infranktpc}
	With the definitions of $E$, $F$, $x$, and $t$ above, $\mathcal{O}^{++}$ is an $\mf{sl}_{\ZZ}$-TPC of type $(\uinfty,\ud)$.
	\end{infranktpc}

\begin{proof}
	\ref{CA1} and \ref{CA3} are clear. \ref{CA2} follows from truncation and Theorem~\ref{finranktpc}. \ref{TPC1} and \ref{TPC2} are consequences of \eqref{FDeltamult} and \eqref{EDeltamult} and the identification of $X^{++}$ with $\Xi_{\uinfty,\ud}$ described in \S~\ref{subsec:Set-up}. \ref{CA4} is a consequence of \ref{TPC2}.
	\end{proof}
	
Recall that the category $\catO^{++}$ depends on $\parity\in\{0,1\}$. We reintroduce the $\parity$-dependence and write $\mathcal{O}^{++}$ as $\catO^{++}_{\parity}$. A less general formulation of the following was conjectured in \cite[Conjecture~6.10]{CWZ06}. It was generalized in \cite[Conjecture~4.18]{CW07} and first proved in \cite[Theorem~5.1]{CL09}.

\begin{superduality}\label{superduality}
	There is a strongly equivariant equivalence
	\begin{align}
		\mathbb{S}:\mathcal{O}_0^{++}\longrightarrow \mathcal{O}_1^{++}
		\end{align}
		
	\noindent with $\mathbb{S}L(\lambda)\cong L(\lambda)$.
	\end{superduality}

\begin{proof}
	This follows immediately from Proposition \ref{prop:TPCsuperduality}, Theorem \ref{thm:uniquenessofTPCs}, and Theorem \ref{infranktpc}. The condition on irreducibles shows that this is the same functor as in \cite{CL09}.
	\end{proof}
	
\section{Graded lifts}\label{section:gradings}

We finish by describing how to construct graded lifts of $\mf{sl}_{I_r}$-TPCs and deduce a graded super duality. This section closely follows \cite[Section~5]{BLW13} and we refer the interested reader there for most definitions and proofs.

For $r\in\NN\cup\{\infty\}$, let $U_q\mf{sl}_{I_r}$ be the quantum group associated to $\mf{sl}_{I_r}$: a $\QQ(q)$-algebra with generators $\left\{ \, e_i, f_i, k_i^{\pm1} \,\middle| \, i\in I_r \, \right\}$ and well-known relations. For $\parity\in\{0,1\}$ there is a $U_q\mf{sl}_{I_r}$-module $\bigwedge_q^{\ur,\ud} V_r$ with basis $\left\{ \, v_{\lambda} \, \middle| \, \lambda\in\Xi_{\ur,\ud}\,\right\}$ as before. The action of the generators on this basis is given in \cite[(5.3)-(5.4)]{BLW13}.

Let $\mc{C}$ be a graded category with grading shift functors $Q^{\pm1}$. Write $\widehat{\mc{C}}$ for the category with the same objects as $\mc{C}$ and Hom spaces defined by
\begin{equation}
	\Hom_{\widehat{\mc{C}}}(M,N):=\bigoplus_{i\in\ZZ}\Hom_{\mc{C}}(Q^iM,N).
	\end{equation}

\noindent A graded functor $F:\mc{C}\to\mc{C}'$ between graded categories induces a functor $\widehat{F}:\widehat{\mc{C}}\to\widehat{\mc{C}'}$. A graded lift of an (ungraded) Schurian category $\overline{\mc{C}}$ is a graded abelian category $\mc{C}$ with a fully faithful functor $\nu:\widehat{\mc{C}}\to\overline{\mc{C}}$ such that $\nu$ is dense on projectives and $\nu\circ\widehat{Q}\cong\nu$.

We define a \emph{$U_q\mf{sl}_{I_r}$-tensor product categorification} of type $(\ur,\ud)$ as in \cite[Definition~5.9]{BLW13}. If $\mc{C}$ is a $U_q\mf{sl}_{I_r}$-TPC of type $(\ur,\ud)$ we denote the distinguished set of irreducibles in $\mc{C}$ by $\left\{\, L(\lambda) \,\middle|\, \lambda\in\Xi_{\ur,\ud}\,\right\}$.

\begin{gradedlifts}\label{gradedlifts}
	Take $r\in\NN\cup\{\infty\}$. Suppose that $\overline{\mathcal{C}}$ is an $\mathfrak{sl}_{I_r}$-TPC of type $(\ur,\ud)$.
	\begin{enumerate}[label=(\roman*)]
		\item \label{gradedlifts1} There exists a graded lift $\mc{C}$ of $\overline{\mc{C}}$ such that $\mc{C}$ is a $U_q\mathfrak{sl}_{I_r}$-TPC $\mathcal{C}$ of type $(\ur,\ud)$ and the graded functors $E_j$ and $F_j$ and graded natural transformations $x$ and $t$ are all graded lifts of the corresponding data for $\overline{\mathcal{C}}$. Moreover $\mathcal{C}$ is Koszul. 
		\item If $\mathcal{C}'$ is another such graded lift of $\overline{\mc{C}}$ then there is a strongly equivariant graded equivalence $\mathbb{G}:\mathcal{C}\to \mathcal{C}'$ with $\nu'\circ\widehat{\mathbb{G}}\cong \nu$ and $\mathbb{G} L\left(\lambda\right)\cong L'\left(\lambda\right)$.
		\end{enumerate}
	\end{gradedlifts}
	
\begin{proof}
	For $r<\infty$ this is covered by \cite[Theorem~5.11]{BLW13}. For $r=\infty$ the proof is formally identical to that outlined in \cite[\S5.7 and \S5.8]{BLW13}. The only small modification needed is the definition of defect. Suppose $\lambda\in\Xi_{\uinfty,\ud}$ and take $r\in\NN$ such that $\lambda\in\Xi_{\ur,\ud}$. Define the defect of $\lambda$ by
	\begin{equation}
		\text{def}(\lambda):=\frac{1}{2}(\lvert \kappa_r\rvert\cdot\lvert\kappa_r\rvert-\lvert\lambda\rvert_r\cdot\lvert\lambda\rvert_r)
		\end{equation}
	
	\noindent where $\kappa_r\in\Xi_{\ur,\ud}$ is defined as in \S\ref{subsection:truncation}. The proof of \cite[Lemma~2.2]{BLW13} can easily be adapted to show that this is independent of the choice of $r$.
	\end{proof}
	
Recall the $\mathfrak{sl}_{\mathbb{Z}}$-TPC ${\mathcal{O}}_{\parity}^{++}$ of type $(\uinfty,\ud)$ from Section~\ref{section:catO} and the super duality equivalence ${\mathbb{S}}:{\mathcal{O}}_{0}^{++}\to {\mathcal{O}}_1^{++}$ of Theorem~\ref{superduality}. The following new `graded super duality' is an immediately corollary of the theorem above.

\begin{gradedsuperduality}
	For $\parity\in\left\{0,1\right\}$, the category ${\mathcal{O}}_{\parity}^{++}$ has a unique Koszul graded lift $^{\text{gr}}\mathcal{O}_{\parity}^{++}$ as in Theorem \ref{gradedlifts}\ref{gradedlifts1} and ${\mathbb{S}}$ lifts to a strongly equivariant graded equivalence
	\begin{align}
		^\text{gr}\mathbb{S}:\,^\text{gr}\mathcal{O}_0^{++}\longrightarrow\, ^\text{gr}\mathcal{O}_1^{++}
		\end{align}
	
	\noindent with $^\text{gr}\mathbb{S} L\left(\lambda\right)\cong L\left(\lambda\right)$.
	\end{gradedsuperduality}

\end{document}